\documentclass[a4paper,12pt,reqno]{amsart}
\usepackage{amssymb}
\usepackage[dvips]{graphicx}
\usepackage{hyperref}
\usepackage{mathtools,amscd}

\input epsf.tex
\newdimen\xsize
\newdimen\oldbaselineskip
\newdimen\oldlineskiplimit
\def\restorelineskip{\baselineskip=\oldbaselineskip%
\lineskiplimit=\oldlineskiplimit}
\def\putm[#1][#2]#3{
\hbox{\vbox to 0pt{\parindent=0pt%
\vskip#2\xsize\hbox to0pt{\hskip#1\xsize $#3$\hss}\vss}}}%
\long\def\Line#1{\hbox to \hsize{#1}}
\def\putt[#1][#2]#3{
\vbox to 0pt{\noindent\hskip#1\xsize\lower#2\xsize%
\vtop{\restorelineskip#3}\vss}}

\makeatletter
\def\xbig[#1]#2{{\hbox{$\m@th\left#2\vbox to#1\xsize{}%
\right.\n@space$}}}
\makeatother
\def\xlar[#1]#2{%
\smash{\mathop{ \hbox to #1\xsize{\leftarrowfill}}\limits^{#2}}}
\def\xrar[#1]#2{%
\smash{\mathop{ \hbox to #1\xsize{\rightarrowfill}}\limits^{#2}}}
\def\xline[#1]{\hbox to #1\xsize{\leaders\hrule\hfill}}

\DeclareFontFamily{U}{rsf}{\skewchar\font'177}%
\DeclareFontShape{U}{rsf}{m}{n}{<-6>rsfs5<6-8>rsfs7<8->rsfs10}{}%
\DeclareFontShape{U}{rsf}{b}{n}{<-6>rsfs5<6-8>rsfs7<8->rsfs10}{}%
\DeclareMathAlphabet\RSFS{U}{rsf}{m}{n}
\SetMathAlphabet\RSFS{bold}{U}{rsf}{b}{n}
\def\sf#1{{\mathsf{#1}}}

\def\slsf{\slshape \sffamily }

\def\msmall#1{\mathchoice{\hbox{\small$\displaystyle {#1}$}}{#1}{#1}{#1}}

\def\cc{{\mathbb C}}

\def\nn{{\mathbb N}}

\def\pp{{\mathbb P}}

\def\loc{{\sf{loc}}}

\def\dim{\sf{dim}\,}

\def\id{\sf{Id}}
\def\im{\sf{Im}\,}

\def\ker{\sf{Ker}\,}
\def\lim{\mathop{\sf{lim}}}

\def\ord{\sf{ord}}

\def\Sing{\sf{Sing}\,}

\def\supp{\sf{supp}\,}

\def\vect{{\mathrm{v}}}

\def\eps{\varepsilon}

\def\<{\langle}\let\la=\<
\def\>{\rangle}\let\ra=\>
 \let\bs=\bss 
\def\comp{\Subset}
\def\d{\partial}
\def\dbar{{\barr\partial}}

\def\ddef{\mathrel{{=}\raise0.3pt\hbox{:}}}
\def\deff{\mathrel{\raise0.3pt\hbox{\rm:}{=}}}
\def\inv{^{-1}}

\def\fraction#1/#2{\mathchoice{{\msmall{ #1\over#2}}}%
{{ #1\over #2 }}{{#1/#2}}{{#1/#2}}}
\def\norm#1{\left\Vert{#1}\right\Vert}

\def\lan{\langle}
\def\ran{\rangle}
\def\le{\leqslant}

\let\xrar=\xrightarrow
\def\emptyset{\varnothing}

\def\longpoints{\leaders\hbox to 0.5em{\hss.\hss}\hfill \hskip0pt}
\def\stateskip{\smallskip}
\def\state#1. {\stateskip\noindent{\bf#1. }} 
\def\statep#1. {\stateskip\noindent{\bf#1 }} 
\def\proof{\state Proof. \2}

\def\Chi{\raise 2pt\hbox{$\chi$}}

\def\ie{\hskip1pt plus1pt{\sl i.e.\/,\ \hskip1pt plus1pt}}

\def\sli{{\sl i)} } 
\def\slii{{\sl i$\!$i)} }

\def\barr#1{\mskip1mu\overline{\mskip-1mu{#1}\mskip-1mu}\mskip1mu}
\def\Chi{\raise 2pt\hbox{$\chi$}}

\let\phI=\phi\let\phi=\varphi\let\varphi=\phI
\def\bfe{{\boldsymbol e}}
\def\bftheta{{\boldsymbol\theta}}%
\def\bfs{{\boldsymbol s}}
\let\cal=\mathcal
\def\cala{{\cal A}}

\def\cale{{\cal E}}
\def\calf{{\cal F}}
\def\calg{{\cal G}}

\def\calk{{\cal K}}

\def\calo{{\cal O}}

\def\eps{\varepsilon}

\def\bs{\backslash}

\def\comp{\Subset}
\def\d{\partial}
\def\dbar{{\barr\partial}}
\def\1{{1\mkern-5mu{\rom l}}}

\def\ge{\geqslant}
\def\inv{^{-1}}

\let\wt=\widetilde
\def\fraction#1/#2{\mathchoice{{\msmall{ #1\over#2}}}%
{{ #1\over #2 }}{{#1/#2}}{{#1/#2}}}

\def\le{\leqslant}

\def\emptyset{\varnothing}
\newcommand{\2}{\thinspace}

\def\ti#1{{\tilde{#1}}}

\def\qed{\ \ \hfill\hbox to .1pt{}\hfill\hbox to .1pt{}\hfill $\square$\par}

\def\comment#1\endcomment{}

 \belowdisplayskip=\abovedisplayskip

\def\lineeqqno(#1){\hfill\llap{\vbox to 10pt%
{\vss\begin{align} \eqqno(#1)\end{align}\vss}}\vskip1pt}

\advance\headheight 1.2pt

\def\ShowwLLabel#1{}

\def\thechpt{\Roman{chpt}}

\def\newchapt[#1]#2{\newpage%
\refstepcounter{chpt}\setcounter{subsection}{0}%
\setcounter{thm}{0}\setcounter{defi}{0}%
\setcounter{rema}{0}\setcounter{exrc}{0}%
\renewcommand{\thesubsection}{\thechpt.\arabic{subsection}}%
\section*{\begin{center}\huge \bf Chapter \thechpt\\
#2 \end{center}}\label{#1}%
\ \smallskip%
\markboth{Chapter \thechpt}{#2}%
}
\def\newsect[#1]#2{\refstepcounter{section}\setcounter{equation}{0}%
\renewcommand{\thesubsection}{\arabic{section}.\arabic{subsection}}%
\section*{\arabic{section}.
#2}\vspace{-20pt}\label{#1}\vspace{20pt}%
\markboth{Section \arabic{section}}{#2}}

\def\newlect[#1]#2{\refstepcounter{section}%
\renewcommand{\thesubsection}{\arabic{section}.\arabic{subsection}}%
\section*{Lecture \arabic{section}\\
#2}\label{#1}%
\markboth{Lecture \arabic{section}}{#2}}

\def\newprg[#1]#2{\refstepcounter{subsection}%
\subsection*{{\thesubsection.\ #2}} \label{#1}%
}

\def\newappx[#1]#2{%
\refstepcounter{appx}\setcounter{section}{0}%
\renewcommand{\thesubsection}{A\arabic{appx}.\arabic{subsection}}%
\section*{Appendix \arabic{appx}\\ #2}
\label{#1}%
\markboth{Appendix A\arabic{appx}}{#2}
}

\newtheorem{thm}{Theorem}[section]
   \def\newthm#1{\begin{thm}\label{#1}}

\newtheorem{nnthm}{Theorem}
   \def\newthm#1{\begin{nnthm}\label{#1}}

\newtheorem{lem}{Lemma}[section]
   \def\newlemma#1{\begin{lem} \label{#1}}

\newtheorem{prop}{Proposition}[section]
   \def\newprop#1{\begin{prop}\label{#1}}

\newtheorem{nnprop}{Proposition}
   \def\newprop#1{\begin{nnprop}\label{#1}}

\newtheorem{corol}{Corollary}[section]
   \def\newcorol#1{\begin{corol} \label{#1}}

\newtheorem{nncorol}{Corollary}
   \def\newcorol#1{\begin{nncorol} \label{#1}}

\newtheorem{defi}{Definition}[section]
   \def\newdefi#1{\begin{defi} \label{#1}\rm }

\newtheorem{nndefi}{Definition}
   \def\newdefi#1{\begin{nndefi} \label{#1}\rm }

\newtheorem{exmp}{Example}[section]
   \def\newexmp#1{\begin{exmp} \label{#1}\rm }

\newtheorem{exrc}{Exercise}
   \def\newexrc#1{\begin{exrc} \label{#1}\rm }

\newtheorem{rema}{Remark}[section]
   \def\newrema#1{\begin{rema} \label{#1}\rm }

   \def\newrema#1{\begin{rema} \label{#1}\rm }

\def\eqqno(#1){\label{(#1)}}
\def\eqqref(#1){(\ref{(#1)})}

\usepackage[dvipsnames]{xcolor} 
\usepackage{colortbl}

\title{A Thullen-type Extension Theorem for Nakano\\[3pt] 
Semi-Positive Holomorphic Vector Bundles}

\author{V. Del\'ecluse, S. Ivashkovych*}

\address{Universit\'e de Lille, UFR de Math\'ematiques, 59655 Villeneuve d'Ascq, France.}
\email{valentin.delecluse@univ-lille.fr}

\address{Universit\'e de Lille, UFR de Math\'ematiques, 59655 Villeneuve d'Ascq, France.}
\email{serge.ivashkovych@univ-lille.fr}

\thanks{* Partially supported by the  R-CDP-24-004-C2EMPI project.}

\date{\today}

\begin{document}
\begin{abstract}
We prove a Thullen-type extension theorem for holomorphic hermitian 
bundles with Nakano semi-positive curvature. More precisely we 
show that the sheaf of holomorphic sections of such a bundle 
extends as a reflexive coherent analytic sheaf.
\end{abstract}

\maketitle

\newsect[INT]{Introduction}

\newprg[INT.result]{Formulation of the Main Result}
Let $X$ be a complex manifold and $Y$ a complex submanifold of $X$ 
of pure codimension one. Let $G$ be an open set of $X$ which contains 
$X\setminus Y$ and intersects every branch of $Y$. A domain of the 
form $G$ described above is called a Thullen-type domain. Another 
way to see a Thullen-type domain is $G = (X\setminus Y)\cup U$, where
$U$ is an open set intersecting every branch of $Y$. We say that an 
analytic object possesses the Thullen-type extension property if it 
extends from $G$ to $X$ for all such $X,Y,G$ as above. For example, 
holomorphic and meromorphic functions possess the Thullen-type extension 
property, see \cite{Si2}.

\smallskip In this paper we shall prove the Thullen-type extension 
theorem for Nakano semi-positive holomorphic vector bundles. Recall 
that a  holomorphic hermitian vector bundle $E$ over a complex manifold
$X$ is called {\slsf Nakano semi-positive} if the curvature of its 
Chern connection is positive semidefinite, see Definition \ref{nak-pos} below.


\begin{nnthm}
\label{thul-posit} 
Let $E$ be a Nakano semi-positive holomorphic hermitian vector bundle on 
the Thullen-type domain $G$ in a complex manifold $X$. Then there 
exists a reflexive coherent  analytic sheaf $\wt\cale$ on $X$ and
an isomorphism $\widetilde{\Phi}:\wt\cale|_G\to \cale $, where $\cale$ 
is the sheaf of holomorphic sections of $E$. Moreover, for every 
holomorphic section $\ti s$ of $\wt \cale |_G$  its $\wt\Phi$-image  
is $L^2$-bounded on compacts in $X$, \ie $\wt \Phi (\ti s|_{K\cap G})
\in L^2(K\cap G,E)$ for every compact $K\comp X$. 

\end{nnthm}

\smallskip 
The special case when $Y$ has codimension two and $E$ is a line bundle
was proved by Shiffman in \cite{Sh1, Sh2}. The general case was announced 
by Siu in \cite{Si1}, where a proof for $\dim X=2$ was sketched. In this
case a bundle extends as a bundle: \ie $\wt\cale$ in Theorem 
\ref{thul-posit} is the sheaf of holomorphic sections of a holomorphic
bundle. 

\newprg[INT.proof]{Main ingredients of the proof.}

Let us outline the three principal steps in the proof of Theorem \ref{thul-posit}.

\medskip\noindent{\bf 1.} First we adapt the result of H\"ormander and Demailly,
namely Theorem 4.1 from \cite{Dm1}, to our situation and explain a solution 
of the  $\dbar$-problem with $L^2$-estimates  in holomorphic hermitian vector 
bundles over  {\slsf non-complete} K\"ahler manifolds.

\smallskip\noindent{\bf 2.} Second we exploit the idea of Siu, 
outlined in \cite{Si1}, how to produce sufficiently many holomorphic 
sections of a bundle with semi-positive curvature, see Theorem \ref{L2-section}
in subsection \ref{POS.sect}. 

\smallskip\noindent{\bf 3.} Finally in Section \ref{EXT} we develop methods
permitting us to increase the dimension of the ambient manifold $X$. 

\smallskip\noindent{\slsf Acknowledgments.} Authors would like to express their
gratitude to V. Shevchishin for the numerous useful discussions on the subject
of this paper. 
The second author would like to thank Y.-T. Siu who pointed to him the fact 
that the case of dimension $\ge 3$ of the ambient manifold $X$ doesn't
follow readily from \cite{Si1}.

\newsect[POS]{Holomorphic bundles with semi-positive curvature}

\newprg[POS.hodge]{Hodge calculus in hermitian vector bundles.}
Let $X$ be a complex hermitian manifold, \ie complex manifold with a 
hermitian metric $h$, and let $(E,h_E)$ be a holomorphic hermitian 
vector bundle on $X$. Denote by $\cala^k(X,E)$ the space of $E$-valued 
smooth $k$-forms and by $\cala^{p,q}(X,E)$ the space of $E$-valued 
$(p,q)$-forms on $X$. The space $\cala^k(X,X\times \cc)$, resp. 
$\cala^{p,q} (X,X\times \cc)$, of smooth $k$-form, resp. of smooth
 $(p,q)$-forms, we denote by $\cala^k(X)$, resp. by $\cala^{p,q} (X)$. 

\smallskip Since every 
hermitian metric $h$ on $X$ is determined by the associated $(1,1)$-form
$\omega = -Im (h)$ we write $h_\omega$ in this situation. If in local 
coordinates metric has representation $h_\omega=\sum h_{jk}dz_j\otimes 
d\bar z_k$ then $\omega = - Im(h_\omega)= - \frac{i}{2}\sum h_{jk}dz_j
\land d\bar z_k$. One uses metrics $h_\omega$ and $h_E$ to define various 
differential-geometric entities on $E\to X$. For example, the volume form 
$dV=\frac{\omega^n}{n!}$ and the Hodge operator $*:\cala^{p,q}(X,E) \to 
\cala^{n-p,n-q}(X,E)$. As usual,
for any operator $S$ between hermitian vector spaces or bundles we denote by
$S^*$ its adjoint  and by $\barr S$ its complex conjugate, provided the
latter is well-defined.  Furthermore, define operators 
$L=L_\omega$, $\Lambda=\Lambda_\omega$, $\tau=\tau_\omega$ on the spaces of 
smooth $(p,q)$-forms with values in $E$ by setting $L_\omega(\alpha):= 
\omega\land \alpha$, $\Lambda_\omega(\alpha):=*^{-1} L_\omega *\alpha$ so that
$\Lambda_\omega=L_\omega^*$, and $\tau_\omega:=[\Lambda_\omega,\partial\omega]$,
where $[\cdot,\cdot]$ denotes the commutator of operators. The latter is 
defined as $[A,B] \deff AB-(-1)^{\ord A\cdot \ord B}BA$. By 
$\partial\omega$ we denote the operator 
$\alpha\in\cala^k(X,E)\mapsto\partial\omega\land\alpha\in \cala^{k+3}(X,E)$
of the left multiplication by $\d\omega$ and $(\partial\omega)^*:
\cala^{k+3}(X,E)\to\cala^k(X,E)$ we its adjoint.  More generally,
if $\eta$ is a smooth $l$-form, we 
denote still by $\eta$ the operator $\cala^k(X,E)\to \cala^{k+l}(X,E)$
of the left multiplication by $\eta$.

\smallskip 
Let $D_E:\cala^0(X,E)\to \cala^1(X,E)$ be the operator of the 
Chern connection in $E$ as well as its  {\slsf exterior extension} 
$D_E:\cala^k(X,E)\to\cala^{k+1}(X,E)$ to $E$-valued differential forms. 
The latter is defined by requiring the Leibnitz rule:
\[
D_E\left(\alpha \otimes s\right) \deff d\alpha\otimes s + 
(-1)^k\alpha\wedge D_Es,
\]
where $\alpha$ is a $k$-form on $X$ and $s$ is a section of $E$. 
Set $\partial_E\deff D^{1,0}_E$, $\dbar_E\deff D^{0,1}_E$, and let 
$\partial^*_E$, $\dbar^*_E$ be the adjoint operators. The
curvature of the Chern connection in $(E,h_E)$ is defined as $F_E:=D_E^2=\partial_E\dbar_E+\dbar_E\partial_E$. Besides this, one 
introduces the Laplace operators $\Delta'_{E}:=\partial_E\partial_E^*+\partial_E^*\partial_E$,
$\Delta'_{E,\tau}:=(\partial_E+\tau)(\partial_E^*+\tau^*)+(\partial_E^*+\tau^*)(\partial_E+\tau)$, and $\Delta''_{E}:=\dbar_E\dbar_E^*+\dbar_E^*\dbar_E$. 
\begin{lem}
\label{Hodge-in-VB}
The following identities hold:
\begin{eqnarray*}
&[\dbar_E^*,L]=i(\partial_E+\tau),\qquad  &[\partial_E^*,L]=-i(\dbar_E+\bar\tau)
\\
&[\dbar_E,\Lambda]=i(\partial^*_E+\tau^*),\qquad
&[\partial_E,\Lambda]=-i(\dbar_E^*+\bar\tau^*).
\end{eqnarray*}
And consequently
\begin{equation}
\label{Delta-E}
\Delta''_E=\Delta'_{E,\tau} + [iF_E,\Lambda]+T_\omega ,
\end{equation}
where
\begin{equation}\label{T-ohm}
T_\omega=\Big[\Lambda,[\Lambda, \tfrac{i}{2}\partial\dbar\omega]\Big]
-\partial\omega\circ(\partial\omega)^*
-(\partial\omega)^*\circ\partial\omega .
\end{equation}
\end{lem}

\begin{corol}
\label{kah-lap}
In particular,  if $\omega$ is K\"ahler, then $\tau = 0$ and 
$T_{\omega} = 0$, and therefore
\begin{equation}
\label{Delta-E-Kh}
\Delta''_E=\Delta'_{E} + [iF_E,\Lambda]
\end{equation}
\end{corol}

\smallskip For the proof we refer to \cite{Dm2}, see also \cite{Dm3}, 
Ch.\,VII. Let $(z_1,\ldots,z_n)$ be a local complex coordinate system 
on $X$.

\begin{lem}
\label{id-lem1} 
Set $\omega_{a,b}:=\frac{i}{2}dz_a\land d\bar z_b$ and
 $u_c:=dz_1\land\ldots\land dz_n\land d\bar z_c$. Then
\begin{equation}
\eqqno(om-lmbd)
[\omega_{a,b}, \Lambda](u_c)= {\frac{1}{2}} \lan d\bar z_c, d\bar z_a \ran 
\cdot u_b.
\end{equation}
\end{lem}
\proof First, assume that $(X, \omega) = (\cc^n, \omega_{euc})$, 
where $\omega_{euc} = \sum_{k=1}^n \omega_{k,k}$ is the 
euclidean form. Define the operators $m_k, \bar m_k$ by
\begin{equation*}
m_k(\alpha) = dz_k \wedge \alpha, \quad \bar m_k(\alpha) = 
d\bar z_k \wedge \alpha,
\end{equation*}
and  by $i_k$ and $\bar i_k$ their adjoints, \ie 
\begin{equation}
\label{def-ik}
i_k = m_k^*, \quad \bar i_k = \bar m_k^*.
\end{equation}
For any 
multiindices $J, K\subset \{ 1, \dots, n\}$ one can show that: 

\medskip\sli if $k\not\in J$, then
\begin{equation}
\eqqno(k-not-in-J)
    i_k(dz_k \wedge d z_J \wedge d\bar z_K) = 2dz_J \wedge d\bar z_K, \quad \text{and} \quad i_k(d z_J \wedge d\bar z_K) = 0,
\end{equation}

\smallskip\slii if $k\not\in K$, then
\begin{equation}
\eqqno(k-not-in-K)
    \bar i_k(d \bar z_k \wedge d z_J \wedge d\bar z_K) = 2dz_J \wedge d\bar z_K, \quad \text{and} \quad \bar i_k(d z_J \wedge d\bar z_K) = 0.
\end{equation}

\medskip  Indeed, since that $\left<dz_k,dz_k\right>=2$ one obtains
for any multi-indices $|L|=|J|$ and $|M|=|K|$ 
\begin{eqnarray*}
& \left<i_k(dz_k \wedge d z_J \wedge d\bar z_K), dz_L\wedge dz_M\right>
= \left<dz_k \wedge d z_J \wedge d\bar z_K, dz_k\wedge dz_L\wedge dz_M\right> = 
\\[5pt]
& = 2\left<d z_J \wedge d\bar z_K, dz_L\wedge dz_M\right>.
\end{eqnarray*}

Since $L$ and $M$ where arbitrary this gives the first identity in 
\eqqref(k-not-in-J). Other identities can be obtained similarly, see \cite{GH}.

\medskip Notice that in our setting $L = \frac{i}{2} \sum_{k=1}^n m_k 
\bar m_k$ and therefore $\Lambda = L^* = -\frac{i}{2} \sum_{k=1}^n \bar 
i_k i_k$. This implies

\begin{eqnarray}
&\Lambda u_c = -\frac{i}{2} \sum_{k=1}^n\bar i_ki_k(u_c) = 
-\frac{i}{2} \sum_{k=1}^n(-1)^{k-1}\bar i_ki_k(dz_k\wedge dz^{(k)}\wedge
d\bar z_c)=
\\[5pt]
& -i\sum_{k=1}^n(-1)^{k-1}\bar i_k(dz^{(k)}\wedge d\bar z_c) = 
-i\sum_{k=1}^n(-1)^{k-1 + n-1}\bar i_k(d\bar z_c\wedge dz^{(k)})=
\\[5pt]
& = -2i(-1)^{n+c} dz^{(c)},
\end{eqnarray}
because due to \eqqref(k-not-in-K) one has 
\begin{equation*}
\bar i_k(d\bar z_c\wedge dz^{(k)}) = 
\begin{cases}
0 \text{ if } k\not= c,\cr
2dz^{(c)} \text{ if } k=c .
\end{cases}
\end{equation*}
 
Therefore

\begin{equation}
\omega_{a,b} \Lambda (u_c) =  (-1)^{n+c} dz_a\wedge d\bar z_b\wedge 
dz^{(c)}         = 
\begin{cases}
0 \text{ if } a\not=c,\cr
dz\wedge d\bar z_b .
\end{cases}
\end{equation}

Furthermore, since $\omega_{a,b}\wedge u_c = \frac{i}{2}dz_a\wedge 
d\bar z_b\wedge dz\wedge d\bar z_c =0$ we see that $\Lambda \omega_{a,b}
\wedge u_c=0$. 

Therefore we obtain that  
\begin{equation*}
    [\omega_{a,b}, \Lambda ](u_c) = \left\{
    \begin{array}{ll}
        0 & \mbox{if } a \neq c \\
        dz_1 \wedge \dots \wedge dz_n \wedge d\bar z_b & \mbox{if } a=c
    \end{array}
\right. = \frac{1}{2} \left<d\bar z_c,d\bar z_a\right>u_b,
\end{equation*}
as stated. This proves the assertion of our lemma  when  
$(X, \omega) = (\cc^n, \omega_{euc})$.

\medskip 
Now, assume that $(X, \omega)$ is a general complex hermitian manifold, 
$x \in X$ and $(z_1, \dots, z_n)$ is a local complex coordinate system
as above. There exists a linear change of coordinates $z = Pw$
such that

\begin{equation}
\eqqno(om-jk)
\omega = \frac{i}{2} \sum_{j,k=1}^n (\delta_{j,k} + O(|w|)) dw_j 
\wedge d\bar w_k
\end{equation}
in coordinates $(w_1, \dots, w_n)$. Since  $dz_j = 
\sum_{k=1}^n p_{jk} dw_k$ we obtain that
\begin{equation}
\omega_{a,b} = \frac{i}{2} \sum_{k,l} p_{ak} \bar p_{bl} dw_k \wedge d\bar w_l,
\end{equation}
and for any $b \in \{1, \dots, n\}$,
\begin{equation}
    u_b = n! \det P \sum_{r=1}^n \bar p_{br} dw \wedge d\bar w_k.
\end{equation}
On the first hand we have
\begin{equation*}
\big\lan d\bar z_c, d\bar z_a \big\ran = \sum_{k,l=1}^n \bar p_{ck}p_{al} 
\big\lan d\bar w_k, d\bar w_l \big\ran,
\end{equation*}
and therefore 
\begin{eqnarray}
&    \left(\big\lan d\bar z_c, d\bar z_a \big\ran u_b\right)_x =
    n! \det P \sum_{k, l=1}^n p_{al} \bar p_{ck} \big\lan d\bar w_k, 
    d\bar w_l \big\ran_x     \sum_{r=1}^n\bar p_{br}  dw \wedge d\bar w_r=
\\[5pt] 
 &   = 2n! \det P \sum_{k,r=1}^n p_{ak} \bar p_{br} \bar p_{ck} dw 
 \wedge d\bar w_r,
\end{eqnarray}
since $\big\lan d\bar w_k, d\bar w_l \big\ran_x = 2\delta_{k,l}$ due
to \eqqref(om-jk). On the other hand we have that
\begin{eqnarray}
&    \left([\omega_{a,b}, \Lambda]u_c\right)_x = n! \det P \sum_{k,l,r=1}^n p_{ak} \bar p_{bl} \bar p_{cr} \left[ \frac{i}{2}dw_k\wedge d\bar w_l, \Lambda \right]_x(dw\wedge d\bar w_r) = 
\\[5pt]
& = n!\det P \sum_{k,l=1}^n p_{ak} \bar p_{bl} \bar p_{ck} dw\wedge d\bar w_l
= n! \det P \sum_{k,l=1}^n\bar p_{ck} ,
\end{eqnarray}
where the second identity is due to the euclidean case: 
$\left[\omega_{k,l}, \Lambda\right]u_r = \frac{1}{2}\left<d\bar w_r,d\bar w_k\right>u_l = \delta_{r,k}u_l$. This proves the assertion of our lemma 
for general $(X,\omega)$.

\smallskip\qed

\begin{lem}
\label{com-nq-to-n1}
Let $\omega_{\alpha, \beta} = \frac{i}{2} dz_{\alpha} \wedge d\bar z_{\beta}$ and $u_I = dz \wedge dz_I$, where $\alpha, \beta \in \{1,\dots, n\}$ and $I=\{i_1,...,i_q)\}\subset \{1,\dots, n\}$ is some multi-index of length $q$. 
Then 
\begin{equation}
\eqqno(eq-com)
    [\omega_{\alpha, \beta}, \Lambda](u_I) = \sum_{a=1}^q (-1)^{a-1}[\omega_{\alpha, \beta}, \Lambda](u_{i_{a}})\wedge d\bar z_{i_1} \wedge \dots \wedge \widehat{d\bar z_{i_a}} \wedge \dots \wedge d\bar z_{i_q}.
\end{equation}
\end{lem}
\proof On the first hand, using the Lemma \ref{id-lem1}, we obtain that the right term of \eqref{(eq-com)} is equal to
\begin{equation*}
    \frac{1}{2}\sum_{a=1}^q (-1)^{a-1}\lan d\bar z_{i_a}, d\bar z_{\alpha}\ran u_{\beta} \wedge d\bar z_{i_1} \wedge \dots \wedge \widehat{d\bar z_{i_a}} \wedge \dots \wedge d\bar z_{i_q},
\end{equation*}
which equals zero if $\alpha \not \in I$. Now, assume that $\alpha \in I$. 
Denote by $l$ the inverse function of $i : \{1, \dots, q\} \to I$. Thus, we have
\begin{eqnarray*}
    &\sum_{a=1}^q (-1)^{a-1}[\omega_{\alpha, \beta}, \Lambda](u_{i_{a}})\wedge d\bar z_{i_1} \wedge \dots \wedge \widehat{d\bar z_{i_a}} \wedge \dots \wedge d\bar z_{i_q} = \\
    &=(-1)^{l_{\alpha}-1}u_{\beta}\wedge d\bar z_{i_1} \wedge \dots \wedge \widehat{d\bar z_{\alpha}} \wedge \dots \wedge d\bar z_{i_q},
\end{eqnarray*}
which equals zero if $\beta \in I\smallsetminus \{\alpha\}$. Now assume that $\beta \not\in I\smallsetminus\{\alpha\}$ and let $J$ be the increasing multiindice defined by $(I \smallsetminus {\alpha})\cup \beta$ and $m$ be the unique increasing bijection from $J$ to $\{1, \dots, q\}$. Hence, 
\begin{equation}
    \sum_{a=1}^q (-1)^{a-1}[\omega_{\alpha, \beta}, \Lambda](u_{i_{a}})\wedge d\bar z_{i_1} \wedge \dots \wedge \widehat{d\bar z_{i_a}} \wedge \dots \wedge d\bar z_{i_q} = (-1)^{l_{\alpha}+m_{\beta}-1} u_J.
\end{equation}
On the other hand, since $\Lambda = -\frac{i}{2}\sum_{k=1}^n \bar i_k i_k$ ($i_k$ and $\bar i_k$ were introduced in \eqref{def-ik}), we obtain using \eqref{(k-not-in-J)} and \eqref{(k-not-in-K)} that
\begin{eqnarray*}
    &\Lambda(u_I) = -\frac{i}{2}\sum_{k=1}^n \bar i_k i_k(dz \wedge d\bar z_I) = -\frac{i}{2}\sum_{k=1}^n (-1)^{k-1} \bar i_k i_k( dz_k \wedge dz^{(k)} \wedge d\bar z_I)
    = \\
    &=-i\sum_{k=1}^n  (-1)^{k-1} \bar i_k (dz^{(k)} \wedge d\bar z_I) = -i\sum_{k \in I}  (-1)^{k-1} \bar i_k (dz^{(k)} \wedge d\bar z_I)=\\
    &=-i\sum_{k \in I}  (-1)^{k+n+l_k} \bar i_k (d\bar z_k \wedge dz^{(k)} \wedge d\bar z_{I\smallsetminus \{k\}})= -2i\sum_{k \in I}  (-1)^{k+n+l_k} dz^{(k)} \wedge d\bar z_{I\smallsetminus \{k\}}.\\
\end{eqnarray*}
Thus, we have 
\begin{equation*}
    [\omega_{\alpha, \beta}, \Lambda](u_I) = (\omega_{\alpha, \beta} \Lambda) (u_I) = \left( \frac{i}{2}dz_{\alpha}\wedge d\bar z_{\beta} \right) \wedge \left(-2i\sum_{k \in I}  (-1)^{k+n+l_k} dz^{(k)} \wedge d\bar z_{I\smallsetminus \{k\}} \right).
\end{equation*}
Hence if $ \alpha \not \in I$, then $[\omega_{\alpha, \beta}, \Lambda](u_I) = 0$. Now assume that $\alpha \in I$. We have 
\begin{equation*}
    [\omega_{\alpha, \beta}, \Lambda](u_I) = (-1)^{\alpha +n+ l_{\alpha}} dz_{\alpha} \wedge d\bar z_{\beta} \wedge dz^{(\alpha)} \wedge d\bar z_{I\smallsetminus {\alpha}}.
\end{equation*}
Thus if $\beta \in I\smallsetminus \{\alpha\}$, then $[\omega_{\alpha, \beta}, \Lambda](u_I) = 0$. Now assume that $\beta \not\in I\smallsetminus\{\alpha\}$, then 
\begin{equation*}
    [\omega_{\alpha, \beta}, \Lambda](u_I) = (-1)^{l_{\alpha}+m_{\beta} -1} dz \wedge d\bar z_J.
\end{equation*}
This finish the proof of the formula \eqqref(eq-com).

\smallskip\qed

\newprg[POS.nakano]{Nakano positivity of holomorphic hermitian bundles.}
The curvature form of a holomorphic hermitian bundle is anti-hermitian and 
therefore $iF_E$ is hermitian and can be expressed in local coordinates 
$z_1,...,z_n$ and a local smooth orthonormal frame $e_1,...,e_r$ of $E$ as 
\begin{equation}
\eqqno(curv-frm1)
iF_E= i\sum_{j,k=1}^n \sum_{\lambda,\mu=1}^r c_{jk\lambda\mu}dz_j\wedge d\bar z_k\otimes e_{\lambda}^*\otimes e_\mu , 
\end{equation}
where $c_{jk\lambda\mu}=\bar c_{kj\mu \lambda}$. 
This allows us to construct a natural hermitian form $\theta_E$ on the 
tensor product $TX\otimes E$. Namely one sets
\begin{equation}
\eqqno(curv-frm2)
\theta_E = \sum_{j,k=1}^n \sum_{\lambda,\mu=1}^rc_{jk\lambda,\mu}
\left(dz_j\otimes e^*_\lambda \right) \otimes 
\overline{\left(dz_k\otimes e^*_{\mu}\right)}.
\end{equation}

Since $TX$ is generated locally by $\frac{\d}{\d z_1}, ...,\frac{\d}{\d z_n}$
one evaluates the form $\theta_E$ on vectors $u=\sum_{j\lambda} 
u_{j\lambda}\frac{\d }{\d z_j}\otimes e_{\lambda}$ and $ v=\sum_{k\mu} 
v_{k\mu}\frac{\d }{\d z_k}\otimes e_\mu \in TX\otimes E$ as follows
\begin{equation}
\eqqno(nak-tens)
\theta_E(u,v) = \sum_{j,k =1}^n\sum_{\lambda,\mu =1}^rc_{jk\lambda\mu}
u_{j\lambda}\bar v_{k\mu}.
\end{equation}

Recall the following 

\begin{defi} \rm
\label{nak-pos}
A holomorphic hermitian bundle $(E,h_E)$ with curvature $F_E$ on a complex 
manifold $X$ is called {\slsf Nakano semi-positive} if for every smooth
section $u$ of the bundle $T'X\otimes E$ one has $\theta_E(u,u)\geq 0$,
\ie at every point $x\in X$ one has 
\begin{equation}
\eqqno(nak-pos1)
\theta_E(u,u)_x = \sum_{j,k =1}^n\sum_{\lambda,\mu =1}^rc_{jk\lambda\mu}
u_{j\lambda}(x)\bar u_{k\mu}(x) \geq 0.
\end{equation}
\end{defi}

Positivity is a pointwise condition, and it can be pointwisely reformulated 
as follows. 

\begin{lem}
\label{pos-herm}
If $(E,h_E)$ is Nakano semi-positive. Then for every  smooth $E$-valued 
$(n,q)$-form $u$ on $X$, $1\le q\le n$, and every $x\in X$ one has
\begin{equation}
\eqqno(nak-pos2)
 \lan \left[iF_E, \Lambda\right]u, u\ran_x \ge 0.
\end{equation}
I.e., the hermitian form $ \lan \left[iF_E, \Lambda\right]u, v\ran_x$ 
on the space $\cala^{n,q} (X, E)$ is positive semidefinite.
\end{lem}
\proof  We must show that the function $\lan\left[iF_E, \Lambda\right]u, u\ran$ is non-negative for any smooth  $(n,q)$-form $u$. 
Fix some $x \in X$ and let $(z_1, \dots, z_n)$ be local coordinates 
centred at $x$ such that
\[
\omega_x = \frac{i}{2} \sum_{k=1}^n dz_k \wedge d\bar z_k, 
\]
and write $iF_E$ as in \eqqref(curv-frm1) in these coordinates,
where $(e_1, \dots, e_r)$ is a local holomorphic frame, which is an 
orthonormal basis of $E_x$. As in the proof of Lemma \ref{id-lem1}
we have that $L = \frac{i}{2} \sum_{k=1}^n m_k \bar m_k$ and 
therefore $\Lambda = L^* = -\frac{i}{2} \sum_{k=1}^n \bar i_k i_k$.

\medskip  Let 
\begin{equation*}
u = \sum\limits_{|J| = q}{'}\, \sum_{\lambda=1}^r u_{J, \lambda} 
dz_1\wedge \dots \wedge dz_n \wedge d\bar z_J \otimes e_{\lambda}
\in \cala^{n,q}(X, E),
\end{equation*}
where $\sum'$ means that the summation is done for increasing multi-indices. In the following, we denote $dz_1\wedge \dots \wedge dz_n$ by $dz$ and set $dz^{(k)} = dz_1 \wedge \dots \wedge \widehat{dz_k} \wedge \dots \wedge dz_n$, where $\widehat{dz_k}$ means that we omit the $k$-th term.
Using \eqref{(k-not-in-J)} and \eqref{(k-not-in-K)}, we obtain
\begin{eqnarray*}
    &\Lambda(u) = -\frac{i}{2}  \sum{}'_{|J| = q}\sum_{\lambda=1}^r u_{J, \lambda}\sum_{h=1}^n \bar i_h i_h(dz \wedge d\bar z_J)\otimes e_{\lambda} = \\
    &=-\frac{i}{2}\sum{}'_{|J| = q} \sum_{\lambda=1}^r u_{J, \lambda}\sum_{h=1}^n (-1)^{h-1} \bar i_h i_h( dz_h \wedge dz^{(h)} \wedge d\bar z_J)\otimes e_{\lambda}
    = \\
    &=-i\sum{}'_{|J| = q} \sum_{\lambda=1}^r u_{J, \lambda}\sum_{h=1}^n  (-1)^{h-1} \bar i_h (dz^{(h)} \wedge d\bar z_J)\otimes e_{\lambda} = \\
    &=-i\sum{}'_{|J| = q} \sum_{\lambda=1}^r u_{J, \lambda}\sum_{l=1}^q  (-1)^{j_l-1} \bar i_{j_l} (dz^{(j_l)} \wedge d\bar z_J)\otimes e_{\lambda}=\\
    &= -i\sum{}'_{|J| = q} \sum_{\lambda=1}^r u_{J, \lambda}\sum_{l=1}^q  (-1)^{n+j_l+l-1} \bar i_{j_l} (d\bar z_{j_l} \wedge dz^{(j_l)} \wedge d\bar z_{J\smallsetminus \{j_l\}}) \otimes e_{\lambda}=\\
    &= -2i\sum{}'_{|J| = q} \sum_{\lambda=1}^r u_{J, \lambda}\sum_{l=1}^q  (-1)^{n+j_l+l-1} dz^{(j_l)} \wedge d\bar z_{J\smallsetminus \{j_l\}}\otimes e_{\lambda},\\
\end{eqnarray*}
where $j_l$ is the $l$-th component of the multi-index $J$. Hence, 
\begin{eqnarray*}
&iF_E\Lambda(u) = \\
&=2\sum{}'_{|J|=q} \sum_{j,k =1}^n \sum_{\lambda, \mu=1}^r c_{jk\lambda\mu} u_{J, \lambda} \sum_{l=1}^q (-1)^{j_l+l} dz_j \wedge dz^{(j_l)} \wedge d\bar z_k \wedge d\bar z_{J\smallsetminus\{j_l\}} \otimes e_{\mu}=\\[5pt]
&= 2 \sum{}'_{|J|=q} \sum_{k =1}^n \sum_{\lambda, \mu=1}^r\sum_{l=1}^q (-1)^{l-1} c_{j_l k\lambda\mu} u_{J, \lambda}  dz \wedge d\bar z_k \wedge d\bar z_{J\smallsetminus\{j_l\}} \otimes e_{\mu}.
\end{eqnarray*}
Now, we generalize the definition of the $u_{I,\lambda}$'s to non-increasing multi-indices $J$ by putting $u_{J, \lambda} = 0$ if $J$ has two identical components and  $u_{\sigma (J), \lambda}=\eps (\sigma)u_{J,\lambda}$ 
for any permutation $\sigma$ of $q$ indices. We obtain, taking into account
that $\Lambda F_E=0$ on $(n,q)$-forms, the following
\begin{eqnarray*}
&\lan \left[iF_E, \Lambda\right]u, u\ran =
\\[5pt]
& = 2 \sum{}'_{|J|, |K|=q} \sum_{k =1}^n \sum_{\lambda, \mu, \nu=1}^r\sum_{l=1}^q (-1)^{l-1} c_{j_l k\lambda\mu} u_{J, \lambda} \bar u_{K, \nu} \lan dz \wedge d\bar z_k \wedge d\bar z_{J\smallsetminus\{j_l\}}, dz \wedge d\bar z_K\ran 
\\[5pt]
&  \lan e_{\mu}, e_{\nu} \ran = 2^{n+1} \sum{}'_{|J|, |K|=q} \sum_{k =1}^n \sum_{\lambda, \mu=1}^r\sum_{l=1}^q (-1)^{l-1} c_{j_l k\lambda\mu} u_{J, \lambda} \bar u_{K, \mu} \lan d\bar z_k \wedge d\bar z_{J\smallsetminus\{j_l\}}, d\bar z_K \ran =
\\[5pt]
& = 2^{n+2} \sum{}'_{|J|, |K|=q} \sum_{k =1}^n \sum_{\lambda, \mu=1}^r\sum_{l=1}^q c_{j_l k\lambda\mu} u_{j_l \cup (J\smallsetminus \{j_l\}), \lambda} \bar u_{k\cup (K\smallsetminus \{k\}), \mu} \cdot
\\[5pt]
&\cdot\lan d\bar z_k \wedge d\bar z_{J\smallsetminus\{j_l\}}, d\bar z_k 
\wedge d\bar z_{K\smallsetminus\{k\}} \ran = 2^{n+q+1} \sum{}'_{|L|=q-1} 
\sum_{j, k =1}^n \sum_{\lambda, \mu=1}^r c_{jk\lambda\mu} u_{jL, \lambda} \bar u_{kL, \mu}.
\end{eqnarray*}
Since $E$ is Nakano semi-positive, then for any multi-indices $L$ of length $q-1$ we have
\begin{equation*}
    \sum_{j, k =1}^n \sum_{\lambda, \mu=1}^r c_{jk\lambda\mu} u_{jK, \lambda} \bar u_{kL, \mu} \geq 0,
\end{equation*}
see \eqqref(nak-tens). This means that the function $\lan \left[iF_E, \Lambda\right]u, u\ran$ is non-negative, and so the lemma is proved.

\smallskip\qed

\smallskip 
Denote by $L^2_{(p,q)}(X,E)$ the space 
of $L^2$-integrable $E$-valued $(p,q)$-forms on $X$ with the norm
\[
\norm{u}^2_{L^2_{(p,q)}(X,E)}=\norm{u}^2_{L^2_{(p,q)}}:=
\int_X \lan u, u\ran dV.
\]

Recall that the tensor $T_\omega$ vanishes on K\"ahler  manifolds. So 
combining the previous results we conclude

\begin{lem}
\label{Nakano-strict-bis} Let  $(X,\omega)$ be a  complex K\"ahler manifold of
dimension $n$, $(E,h)$ a Nakano semi-positive holomorphic hermitian vector
bundle,  and $\phi$ a strictly PSH function such that $i\partial\dbar \phi
\geq c\cdot\omega$ with some constant  $c>0$. Define the new metric $h_\phi
:=e^{-\phi}h$ on $E$ and let $F_\phi$ be the  curvature of $(E,h_\phi)$. 
Then
\begin{equation}
\eqqno(dbar-est1-pt-bis)
\big\lan [iF_\phi,\Lambda]u,u\big\ran_x
\geq c\cdot |u|_x^2
\end{equation}
for every $u\in \cala^{n,q}(X,E)$ and any $x \in X$. Consequently, 
\begin{equation}
\eqqno(dbar-est1-bis)
\int_X \big\lan [iF_\phi,\Lambda]u,u\big\ran dV
\geq c\cdot \norm{u}^2_{L^2(X)}
\end{equation}
for every $u\in \cala^{n,q}(X,E)$.
\end{lem}

\proof The curvature of the trivial line bundle $\calo$ equipped with the
metric $|f|^2_\phi:=e^{-\phi}|f|^2$ is $\partial\dbar\phi$. We can see the 
bundle $(E,h_\phi)$ as the tensor product $(E,h)\otimes(\calo,e^{-\phi})$. 
This implies the formula $F_\phi=F_E \otimes \id_\calo + \id_E\otimes 
\partial\dbar\phi $. Fix a point $x \in X$ and a smooth $E$-valued 
$(n,q)$-form $u$. Since $E$ is Nakano semi-positive the
Lemma \ref{pos-herm} implies that 
\begin{equation*}
    \big\lan [iF_\phi,\Lambda]u,u\big\ran_x = \big\lan [iF_E,\Lambda]u,u\big\ran_x + \big\lan [i\d \dbar \phi,\Lambda]u,u\big\ran_x \geq \big\lan [i\d \dbar \phi,
    \Lambda]u,u\big\ran_x.
\end{equation*}
Furthermore, $i\d\dbar \phi \geq c\cdot \omega$. Hence, 
\begin{equation*}
    \big\lan [iF_\phi,\Lambda]u,u\big\ran_x \geq c \cdot 
    \big\lan [\omega,\Lambda] (u),u\big\ran_x.
\end{equation*}
Now fix local coordinates $(z_1, \dots, z_n)$ centred at $x$ such that
\begin{equation}
    \omega_x = \sum_{k=1}^n \omega_{k,k},
\end{equation}
where $\omega_{k,k} = \frac{i}{2} dz_k \wedge d\bar z_k$ and $(e_1, \dots, e_r)$
a local holomorphic frame of $E$ in a neighbourhood of $x$. Then $u$ can be 
written as
\[
u = \sum\limits_{|J| = q}{'} \sum_{\lambda=1}^r u_{J,\lambda} dz \wedge d\bar z_J \otimes e_{\lambda}
\]
Then the Lemma \ref{com-nq-to-n1} implies that
\begin{eqnarray*}
    &\left[\omega , \Lambda\right]u = \sum\limits_{|J| = q}{'} \sum_{\lambda=1}^r \sum_{k=1}^n u_{J,\lambda} \left[\omega_{k,k}, \Lambda\right](dz \wedge d\bar z_J) \otimes e_{\lambda} =
    \\[5pt]
    &= \sum\limits_{|J| = q}{'} \sum_{\lambda=1}^r \sum_{k=1}^n u_{J,\lambda} \sum_{l=1}^q (-1)^{l-1} \left[\omega_{k,k}, \Lambda\right](dz \wedge d\bar z_{j_l})\wedge d\bar z_{J \smallsetminus \{j_l\}} \otimes e_{\lambda} =
    \\
    &= \frac{1}{2}\sum\limits_{|J| = q}{'} \sum_{\lambda=1}^r \sum_{k=1}^n u_{J,\lambda} \sum_{l=1}^q (-1)^{l-1} \lan d\bar z_{j_l}, d\bar z_k \ran dz \wedge d\bar z_{k}\wedge d\bar z_{J \smallsetminus \{j_l\}} \otimes e_{\lambda} =
    \\
    & =\sum\limits_{|J| = q}{'} \sum_{\lambda=1}^r \sum_{k=1}^n u_{J,\lambda} dz \wedge d\bar z_{J} \otimes e_{\lambda} = u.
\end{eqnarray*}
Therefore, we obtain the desired inequalities \eqqref(dbar-est1-pt-bis) 
and \eqqref(dbar-est1-bis).

\smallskip\qed

\newprg[POS.sol]{Solving the $\dbar$-equation with values in
positive bundles.}

 We refer to \cite{Ho} and \cite{Dm3}, Ch.\,VIII, for the introduction to
H\"ormander's $L^2$-method in complex analysis.
Let $(X,\omega)$ and $(E,h_E)$ be as above. Denote by $T_\dbar:L^2_{(p,q)}
(X,E)\to L^2_{(p,q+1)}(X,E)$ the closed with respect to the graph norm 
operator defined by $\dbar_E$. We refer to the online book of Demailly
\cite{Dm3}, see Theorem 3.2 Ch. VIII, for the proof of the following 
statement.


\begin{prop}
\label{var}
Assume that $(X, \omega)$ is a complete hermitian manifold and $E$ a
hermitian vector bundle over $X$. Then for any 
$E$-valued $(p,q)$-form $u \in Dom\left(T_{p, q-1}^*\right) \cap Dom\left(T_{p,q}\right)$ there exist smooth, compactly supported $E$-valued $(p,q)$-forms $\{u_k\}$ such that 
\begin{equation*}
\lim_{k\to \infty} \| u-u_k\| +\|T_{\bar\partial}^* u - \bar\partial^*_E 
u_k \| + \|T_{\bar\partial} u - \bar\partial_E u_k\| = 0.
\end{equation*}
In other words, the set of smooth, compactly supported $E$-valued $(p,q)$-forms is dense in $Dom\left(T_{\dbar}^*\right) \cap Dom\left(T_{\dbar}\right)$ 
with respect to the graph norm.
\end{prop}

\smallskip The following result follows directly from Corollary \ref{kah-lap}
for forms with compact support, and then via 
Proposition \ref{var} in the general case of {\slsf complete} K\"ahler
$X$.

\begin{prop}
\label{Demai-estim}
 Assume that the metric $h_\omega$ is complete. Then
\begin{equation}
\label{dbar-Dm-estim1}
 \norm{\dbar_E u}^2_{L^2(X)}+\norm{\dbar_E{}^*u}^2_{L^2(X)}
\geq
\int_X \big(\big\lan [iF_E,\Lambda]u,u\big\ran +\lan T_\omega u, u\ran\big) dV.
\end{equation}
In particular, when $h_\omega$  is K\"ahler, then
\begin{equation}
\eqqno(dbar-Dm-estim2)
 \norm{\dbar_E u}^2_{L^2(X)}+\norm{\dbar_E{}^*u}^2_{L^2(X)}
\geq
\int_X \big\lan [iF_E,\Lambda]u,u\big\ran  dV.
\end{equation}
\end{prop}

\begin{thm}[\cite{Ho}]
\label{L2-in-VB}
Assume that the metric $h_\omega$ is complete and that for a given bi-index
$(p,q)$ there exists a constant $c>0$ such that
\begin{equation}
\eqqno(dbar-estim)
\norm{\dbar_E u}^2_{L^2(X)}+\norm{\dbar^*_Eu}^2_{L^2(X)}
\geq c\cdot \norm{u}^2_{L^2(X)}
\end{equation}
for every $u\in L^2_{(p,q)}(X,E)$ such that $\dbar_E u\in L^2_{(p,q+1)}(X,E)$ 
and $\dbar^*_Eu\in L^2_{(p,q-1)}(X,E)$. Then for every form $u\in L^2_{(p,q)}(X,E)$ such that $\dbar_Eu=0$  in the weak sense there exists $v\in L^2_{(p,q-1)}(X,E)$ such that $\dbar_Ev=u$ and and  $\dbar_E^*v=0$
in the weak sense and 
\begin{equation}
\eqqno(debar-est1)
c\cdot \norm{v}^2_{L^2(X)}\leq\norm{u}^2_{L^2(X)}.
\end{equation}
\end{thm}

Now we are ready to prove the needed statement. 

\begin{thm}
\label{solve-dbar} Let $X$ be a Stein manifold of dimension $n$,
$\omega$ a K\"ahler form on $X$, $(E,h)$ a Nakano semi-positive 
holomorphic hermitian vector bundle over $X$, and $\phi$ a non-negative 
strictly PSH function such that $i\partial\dbar \phi\geq c\cdot\omega$ 
with some constant $c>0$. Set $h_\phi:=e^{-\phi}h$. Then for every 
$1\leq q \leq n$ and every $u\in L^2_{n,q}(X,E,h_\phi)$ such that 
$\dbar_E u=0$ there exists $v\in L^2_{n,q-1}(X,E,h_\phi)$ such that 
$\dbar_Ev=u$ and $\dbar^*_Ev=0$ in the weak sense and such that
\[
c \cdot\norm{v}^2_{L^2_{n,q-1}(X,E,h_\phi)} \leq \norm{u}^2_{L^2_{n,q}(X,E,h_\phi)}.
\]
\end{thm}

\proof To start with assume at the beginning  that $\omega$ induces 
a complete metric on $X$. Then the result follows from Theorem 
\ref{L2-in-VB}, where the needed estimate \eqqref(dbar-estim) is
obtained from estimates \eqqref(dbar-est1-bis) and 
\eqqref(dbar-Dm-estim2).

\smallskip
The most important part of the theorem is that the result holds also in
the non-complete case. To show this, we first notice that every Stein complex
manifold admits a complete K\"ahler metric. Let $\hat\omega$ be a K\"ahler 
form on $X$ which induces a complete metric on $X$. Then for every 
$\varepsilon>0$ the metric induced by the K\"ahler form 
$\omega_\varepsilon:=\omega+\varepsilon\cdot\hat\omega$ is also complete. 
Here we make the following trivial but important observation: the norm of 
differential forms with respect to $\omega_\varepsilon$ is \emph{smaller} 
than that with respect to $\omega$ (in contrary to tangent vectors where
the situation is just opposite). In particular
\[
\norm{u}^2_{L^2_{n,q}(X,E,h_\phi,\omega_\varepsilon)} \leq \norm{u}^2_{L^2_{n,q}(X,E,h_\phi,\omega)},
\]
and $\omega_\varepsilon$-norms are finite if $\omega$-norms are. Next, find an \emph{exhausting}
strictly PSH function $\psi$ such that $i\partial\dbar\psi\geq c\cdot\hat\omega$. Then $i\partial\dbar(\phi + \varepsilon\cdot\psi)\geq c\cdot \omega_\varepsilon$. Besides, we may assume that $\psi\geq0$ everywhere on
$X$, and hence $0<e^{-\psi}\leq1$. This gives us the estimate
\[
\norm{u}^2_{L^2_{n,q}(X,E,h_{\phi+\varepsilon\cdot\psi},\omega_\varepsilon)} \leq \norm{u}^2_{L^2_{n,q}(X,E,h_\phi,\omega)},
\]

Now by the results of previous steps, for every $\varepsilon>0$ there exists \\
$v_\varepsilon\in L^2_{n,q-1}(X,E,h_{\phi+\varepsilon\cdot\psi},\omega_\varepsilon)$ such that $\dbar_E v_\varepsilon=u$ in the sense of
currents and
\[
c\cdot  \norm{v_\varepsilon}^2_{L^2_{n,q-1}(X,E,h_{\phi+\varepsilon\cdot\psi},\omega_\varepsilon)}  \leq   \norm{u}^2_{L^2_{n,q}(X,E,h_{\phi},\omega)}.
\]
Then for an appropriate subsequence $\varepsilon_\nu$, the sequence $v_{\varepsilon_\nu}$ converges weakly in the $L^2$-sense on every
compact $K\comp X$. Let $v_\infty$ be this limit. Then
\[
c\cdot  \norm{v_\infty}^2_{L^2_{n,q-1}(K,E,h_{\phi},\omega)}  \leq \norm{u}^2_{L^2_{n,q}(X,E,h_{\phi},\omega)}.
\]
for every compact $K\comp X$, and therefore on the whole $X$.

\smallskip%
Finally, each $v_\varepsilon$ satisfies in the weak sense relations
$\dbar_E v_\eps = u$ and $(\dbar_E)^*_\varepsilon v_\varepsilon=0$
in which $(\dbar_E)^*_\varepsilon$ is the adjoint operator to
$\dbar_E$ with respect to the metric $\omega_\varepsilon$. Since 
$\omega_\varepsilon$ converge to $\omega$ in $C^\infty$-topology
we have the same convergence of the coefficients of the operators: 
$(\dbar_E)^*_\varepsilon\longrightarrow\dbar_E^*$ where $\dbar_E^*$ 
is the operator adjoint with respect to $\omega$. Together with the 
weak convergence $v_{\varepsilon_\nu}\longrightarrow v_\infty$ in the
limit we obtain $\dbar_Ev_\infty =u$ and $\dbar_E^*v_\infty=0$ in 
the weak sense.  This proves the theorem. 

\smallskip\qed

\newprg[POS.sect]{Finding  $L^2$-bounded holomorphic sections} 
Consider the following situation. Let $X$ be a Stein manifold of dimension 
$n$ with a K\"ahler form $\omega$, $Y\subset X$ an analytic
set of pure codimension one, $G\subset X$ an open set which contains 
$X\bs Y$ and intersects every branch of $Y$, 
$(E,h)$ a hermitian bundle of rank $r$ over $G$ with
Nakano semi-positive curvature, and $V\Subset X$ a relatively compact Stein 
domain. Denote by $\barr V$ the closure of $V$ in $X$.
The proof of the following theorem follows ideas of \cite{Si1}.

\begin{thm}
\label{L2-section} 
Let $p\in V\cap G$ be a point. Assume that there exists a non-constant 
holomorphic function $f$ on $X$ vanishing at $p$ and non-vanishing on 
$Y\bs G$. Then there exist holomorphic sections  $s_1,\ldots,s_r\in 
L^2(V\cap G,E\otimes\Lambda^{n,0}X)$  which generate the basis of the 
fibre of $E\otimes\Lambda^{n,0}X$ at  $p$.
\end{thm}
\proof Let $Z_f\subset X$ be the zero set of the function $f$. Take a 
Stein neighbourhood $U$ of $Z_f$, for example $U=\{|f(z)|<\eps \}$ with 
$\eps >0$ small enough. Fix a smooth cut-off function $\chi$ on $X$ which 
is identically $1$ in some smaller neighbourhood of $Z_f$ and 
whose support $\supp(\chi)$ lies in $U$.

\smallskip Since $U$ is Stein, every fibre of the bundle 
$E\otimes\Lambda^{n,0}X$ is 
generated by global holomorphic sections. Take any vector $v$ in the fibre 
$(E\otimes\Lambda^{n,0}X)_p$
of this bundle at point $p$ and let $s^*\in\calo(U,E\otimes\Lambda^{n,0}X)$
be a holomorphic section in $U$ such that $s^*(p)=v$. We claim that
$\alpha:=f\inv\cdot\dbar\chi\cdot s^*$ is $L^2$-integrable in $V$. To show this 
let us observe that the set $\barr V \cap \supp(\chi)$ is compact and lies in
$U\subset G$. Consequently, the vector valued form $\dbar\chi\cdot s^*$ is 
$L^2$-integrable in $V$. Furthermore, the set $\barr V \cap \supp(\dbar\chi)$ 
is also compact, and hence the restriction of the function $|f|$ to this
set achieves its minimum at some point $q$. This minimum can
not be $0$, since otherwise $f$ would vanish at $q$ and $q$ would lie on the
set $Z_f$, in contradiction to the fact that $\chi$ is identically $1$ in some
neighbourhood of $Z_f$ and hence $\supp(\dbar\chi)\cap Z_f=\emptyset$. 
Consequently, $f\inv$ is uniformly bounded on $\barr V \cap\supp(\dbar\chi)$ 
and so $\alpha=f\inv\cdot\dbar\chi\cdot s^*$ is $L^2$-integrable in $V$ as 
claimed.

\smallskip Therefore it is integrable in $V\setminus Y$, write this as
$\alpha\in L^2_{n,1}(V\bs Y,E,\omega)$. 
Further, since $f$ is non-vanishing on $V\cap\supp(\dbar\chi)$  the form 
$\alpha=f\inv\cdot\dbar\chi\cdot s^*$ is $\dbar_E$-closed: $\dbar_E\alpha=0$.  Now by Theorem \ref{solve-dbar}, $\alpha=\dbar_E\psi$ for some $\psi\in L^2_{n,0}(V
\bs Y,E,\omega)$. Consequently, $\dbar_E(\chi\cdot s^*-f\cdot\psi)=0$. This
means that $\chi\cdot s^*-f\cdot\psi=:s$ is a holomorphic section of 
$E\otimes\Lambda^{n,0}X$ over $V\bs Y$ which is $L^2$-integrable. By
Riemann's extension theorem, $s$ extends holomorphically to the set $V\cap G$.

\smallskip Next, recall that $\chi$ is identically $1$ in a neighbourhood of the
point $p$. It follows that in a neighbourhood of $p$ the form
$\alpha=f\inv\cdot\dbar\chi\cdot s^*=\dbar_E\psi$ vanishes identically, 
and so $\psi$ is holomorphic near $p$. Thus $s(p)=\chi(p)\cdot s^*(p)-f(p)\cdot\psi(p)=1\cdot s^*(p)-0\cdot\psi(p)=s^*(p)=v$. This
means that $s(z)$ is a $L^2$-integrable holomorphic section of $E\otimes\Lambda^{n,0}X$ in
$V\cap G$ which takes the prescribed value $v$ at the point $p$. This
implies the statement of our Theorem.

\smallskip\qed

\newsect[EXT]{Extension of semi-positive bundles}

\newprg[EXT.ind]{Elimination of indeterminacies of a meromorphic mapping}

Let us  recall now the Hironaka's Resolution Singularities theorem. We shall
quote the so called imbedded resolution of singularities. First we recall  the 
notion of the sequence of blows-up over a complex manifold $D$. Fix a point 
$s_0\in D_0:=D$. Let $l_0$ be a smooth
closed submanifold  of $D_0$ of codimension at least two, passing through $s_0$.
Denote by $\pi _1 : D_1\longrightarrow D_0$ the blow-up of $D_0$ along $l_0$.
Call this {\slsf a  blow-up} of $D_0$ {\slsf along the  centre} $l_0$.
The exceptional divisor $\pi ^{-1}(l_0)$ of this blow-up we denote by $E_1$.
We can repeat this procedure, taking a smooth closed submanifold  $l_1$ in $D_1$
of codimension at least two and such that $\pi (l_1)\ni s_0$.

\begin{defi}
\label{blow-up2}
A finite sequence $\{ \pi ^j\} _{j=1}^N $ of such blows-up we call a sequence
of blows-up over $s_0\in D$ or, a {\slsf regular modification} over $s_0$.
\end{defi}

\smallskip
By $\{ l_j\} _{j=0}^{N-1}$ we denote the corresponding centers and by  $\{ E_j\}_{j=1}^N $
the exceptional divisors.  We put $\pi = \pi_1\circ ...\circ \pi_N$ and
$\hat D=D_N$. $E$ denotes the exceptional divisor of $\pi $, i.e.
$E=\pi_N^{-1} (l_{N-1}\cup ...\cup (\pi_1\circ ... \circ \pi_N)^{-1}(l_0))$.

\begin{thm}
\label{hironaka}
Let $L$ be a subvariety of $D$ and $s_0$ a singular point of $L$. Then 
there exists a regular modification $\pi :\hat D\to D$ over $s_0$ such 
that:

1) the strict transform $\hat L$ of $L$ is smooth in a neighbourhood of 
$\pi^{-1}(s_0)$;

2) $l_i\subset \Sing L_i$, where $L_i$ is a strict transform of $L$ by
$\pi_1\circ ...\circ \pi_i$ and $L_0\deff L$.
\end{thm}
The nice feature of this statement is that the centre of every blowing up is a
smooth subvariety of the {\slsf singular} locus of the variety under smoothing.
For the proof we refer to \cite{Hi1, Hi2}, or to \cite{BM}.

\smallskip Theorem \ref{hironaka} implies the following statement about 
elimination of indeterminacies of meromorphic functions/mappings by 
blows-up. Let $f=h/g$ be a germ of a meromorphic function at $0\in \cc^n$, 
$n\ge 2$. Without loss of generality we assume that holomorphic germs $h$
and $g$ are relatively prime in $_n\calo_0$. Consider the  complex 
hypersurface $C\deff \{z: h(z) \cdot g(z) = 0\}$ 
in a neighbourhood $V$ of zero in $\cc^n$, where $h$ and $g$ are defined.  
By theorem of  Hironaka \ref{hironaka} one can find a proper modification 
$\pi : \hat V \to V$ (which is a composition
of a finite number of consecutive blow-ups with smooth centres), such that the 
proper pre-image $\hat C$ of $C$ under $\pi$ is smooth when restricted to 
a neighbourhood $\widehat W$ of $\pi^{-1}(0)$. Functions 
$\hat h\deff h\circ \pi$ and $\hat g\deff g\circ\pi$ are holomorphic on 
$\hat W$ and don't have common zeroes. Therefore $z\to [\hat h(z):\hat g(z)]$ 
is a holomorphic map from $\widehat W$ to $\pp^1$, or in other words function 
$\hat f\deff f\circ \pi$ has no indeterminacy points in $\widehat W$. 

\smallskip The same reasoning  applies to a meromorphic map with values in
the complex projective space. Every meromorphic mapping $f$ from a manifold
$V$ to complex projective space writes as 
\[
f: z\in V \longmapsto (f_1,...,f_N),
\]
where $f_1,...,f_N$ are meromorphic functions. Shrinking $V$ we can write 
$f_1=h_1/g_1,...,f_N=h_N/g_N$. Resolving indeterminacies of 
every $f_j$ we obtain the holomorphicity of $\hat f\deff f\circ \hat \pi$
for an appropriate proper modification (after an appropriate shrinking). 
To see this set $\hat h_k\deff h_k\circ \pi$ and $\hat g_k\deff g_k\circ \pi$.
They are holomorphic without common zeroes. Let us see that 
$\hat f : z\to (\hat f_1,...,\hat f_N)$ defines a holomorphic map from $\hat V$ to 
$\pp^N$. Indeed, given any point $z_0\in \hat V$ there exists at most one 
$k$ such that $\hat g_k(z_0) = 0$, and then $\hat h_k(z_0)\not = 0$. 
In homogeneous coordinates of $\pp^N$ the (meromorphic) mapping $\hat f$ 
writes as
\[
\hat f : z\to [1:\hat f_1:...:\hat f_k:...:\hat f_N] = [1/\hat f_k:\hat f_1/\hat f_k ...
:1:...:\hat f_N/\hat f_k] \quad\text{ for }\quad z\in \hat V. 
\]
Notice that the second representation is holomorphic in a neighbourhood of $z_0$,
because $1/\hat f_k = \hat g_k/\hat h_k$ has no poles near $z_0$. We conclude
the following:

\begin{prop}
\label{gr-mer-ext}
Let $f :V\to \pp^N$ be a meromorphic mapping. Then for a given relatively compact 
$W\comp V$ there exists a proper modification $\pi :\hat W\to W$ such that the 
lift $\hat f \deff f\circ \pi$ of $f$ to $\hat W$ is holomorphic.
\end{prop}

\begin{rema} \rm 
\label{zero-ind}
We conclude our discussion of meromorphic mappings with two more
observations. 
\begin{itemize}
\item Grassmanian $Gr(r,n)$ is a projective variety, \ie it admits a 
holomorphic embedding to $\pp^N$, see \cite{GH}. And therefore  the 
conclusion of the Proposition \ref{gr-mer-ext} holds true for 
meromorphic mappings with values in $Gr(r,n)$.

\smallskip\item  Given a meromorphic map $f$ in a neighbourhood $V$
of zero with values in $\pp^N$. After shrinking this neighbourhood, 
if necessary, we can write $f$ in the form 
\[
f(z) = [\phi_0(z):\phi_1(z):...:\phi_N(z)],
\]
where all $\phi_j$ are holomorphic and vanish at zero. If they have 
a common divisor in 
$_n\calo_0$ then we can divide them by this divisor without changing 
the map. If at least one of $\phi_j$ doesn't vanish at zero then 
$f$ is holomorphic. Therefore we are left with the following case:
 all $\phi_j$ vanish at zero but they have no common divisor in 
$_n\calo_0$. Then zero is a point of indeterminacy of $f$.
Indeed, suppose that $f$ is holomorphic in a neighbourhood $V$ of zero. 
Shrinking $V$, if necessary we can assume that $f|_V$ takes values in
some affine chart $U_j = \{[z]\in \pp^N:z_j\not=0\}$ of $\pp^N$.
But this means that $\phi_j|_V\not=0$, contradiction.
\end{itemize}
\end{rema}



\newprg[EXT.sub-sh-bis]{Proof of Theorem \ref{thul-posit}, local case}

This proof  will be carried out in several steps. For $\eps \in (0, 1)$ 
set $G^n_\eps = \left(\Delta^{n-1}\times \check\Delta  \right)\cup 
\left(B^{n-1}_\eps\times \Delta\right)$, where $B^{n-1}_\eps$ is the 
ball of radius $\eps$ in $\cc^{n-1}$ centred at zero. Let $(E, h)$ be a 
holomorphic hermitian vector bundle of rank $r$ over $G^n_{\eps}$ with
Nakano semi-positive curvature. Our first aim is to extend $E$ to 
$\Delta^n$ as a reflexive coherent analytic sheaf.

\begin{figure}[h]
\centering
\scalebox{.5}{\includegraphics[width=4.0in]{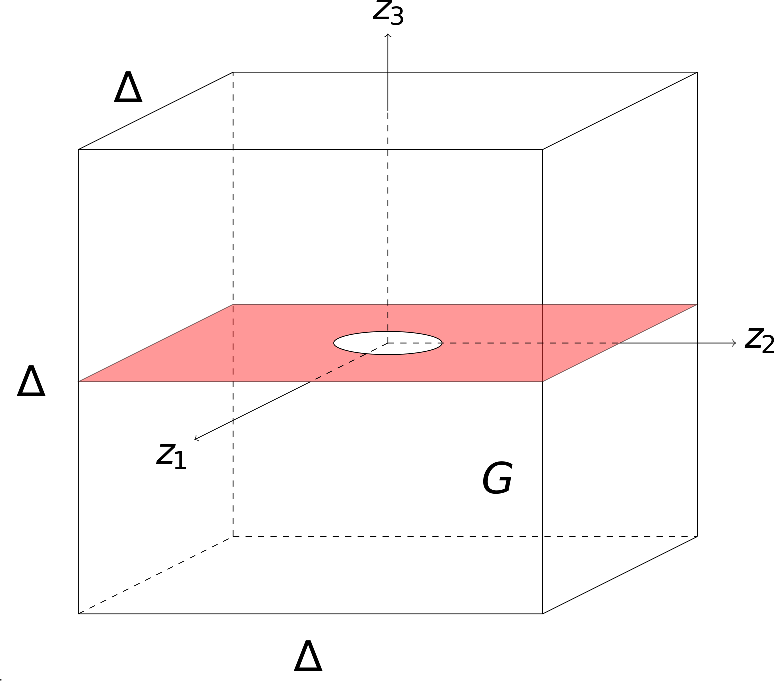}}
\caption{Drawing of $G = G^n_\eps$ in dimension $n=3$.}
\end{figure}


If $n \geq 3$, then for every unit vector $\vect\in \cc^{n-1}$ fix 
some orthonormal coordinate system $z_1,...,z_{n-1}$ in which 
$\vect = (1,0,...,0)$. Denote by $P^n_\eps (\vect)$ the 
open subset of $\cc^n$, which in this new coordinate system
writes as 
\begin{equation}
P^n_\eps (\vect) = \left(\Delta_{\sqrt{n-1}} \times \Delta^{n-2}_{\rho} 
\times \check\Delta\right) \cup \left(\Delta^{n-1}_{\rho} \times 
\Delta\right),
\end{equation}
where $\rho = \frac{\eps}{\sqrt{n-1}}$. The envelope of holomorphy 
of $P^n_{\eps}(\vect)$ in this coordinate system  is
\[
\widehat{P}^n_\eps(\vect)\deff \Delta_{\sqrt{n-1}} \times \Delta^{n-2}_{\rho} \times \Delta.
\]

\smallskip

\begin{itemize}
\item For $n\ge 3$ consider the following $n$-dimensional Thullen figure:
\begin{equation}
\eqqno(thullen-n-bis)
G^n_\eps(\vect) = \Delta^n \cap P^n_{\eps}(\vect).
\end{equation}
\item If $n=2$ we keep the notation
\begin{equation}
\eqqno(thullen-2-bis)
G^2_\eps \deff \Delta\times\check\Delta \cup \Delta_\eps\times \Delta .
\end{equation}
\end{itemize}

\begin{figure}[h]
\centering
\includegraphics[width=4.5in]{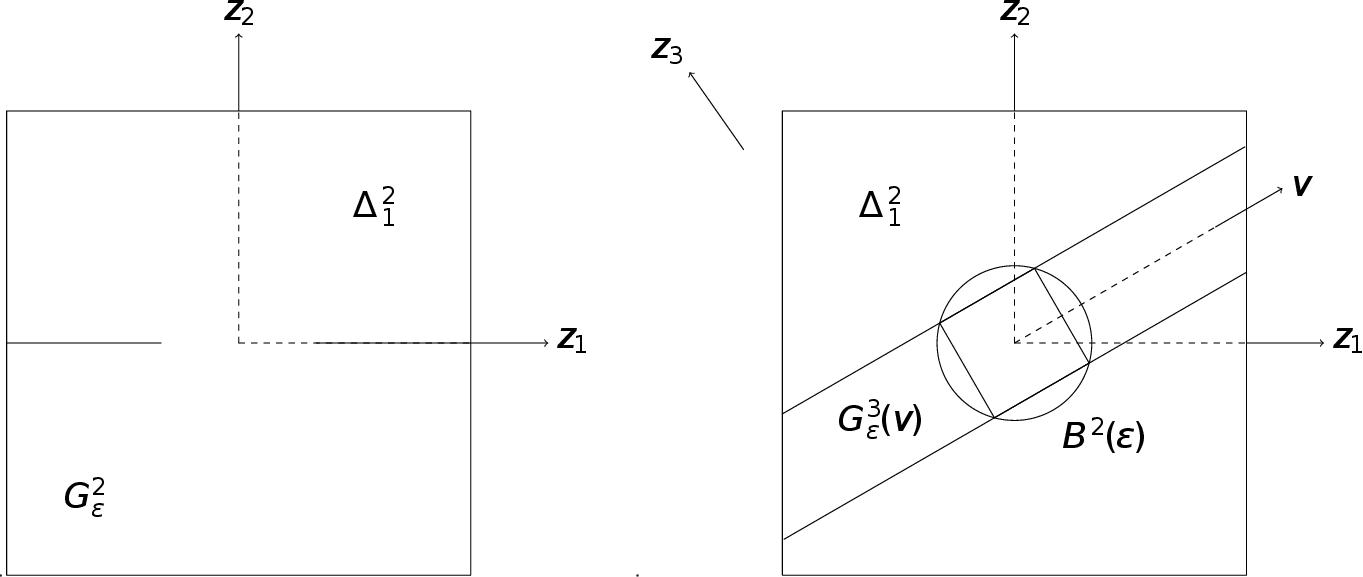}
\caption{Drawing $G^2_{\eps}$ on the left and $G^3_{\eps}(\vect)$ on the 
right.}
\end{figure}

\newpage Notice that the envelopes of holomorphy 
of these  domains are:

\smallskip 
\begin{itemize}
\item for $n=2$ 
\[
\widehat{G}^2_\eps = \Delta^2,
\]

\item for $n\ge 3$
\[
\widehat{G}^n_\eps(\vect) \deff \Delta^n \cap \widehat{P}^n_\eps(\vect).
\]
\end{itemize}

\smallskip
In what follows when speaking about shrunken Thullen domains 
or polydisks we mean that disk $\Delta$ in their definition is replaced
by a smaller disk $\Delta_{1-\delta}$ (as well as $\Delta_\eps$ replaced 
by $\Delta_{\eps -\delta}$), and the ball $B^{n-1}_\eps$ by a smaller 
ball $B^{n-1}_{\eps -\delta}$. Parameter $\delta >0$, small enough, 
is appropriately chosen. Shrunken $G^n_\eps $, $\Delta^n$ and so on 
will be still denoted 
as $G^n_\eps$, $\Delta^n$ and so on if no confusion could occur. By $Y$ 
we denote the hyperplane $\{z_n=0\}$ in $\Delta^n$.

\smallskip\noindent{\slsf Step 1. We state this step in the form of a lemma.} 
In what follows $\vect$ is a fixed unit vector of $\cc^{n-1}$.

\begin{lem}
\label{surj-map1-bis}
Let $E$ be a Nakano semi-positive bundle over $G^n_\eps(\vect)$ 
of rank $r$. Then for an appropriately shrunken $G^n_\eps(\vect)$ there
exists a positive integer $N$ and a holomorphic morphism of bundles  
\begin{equation}
\eqqno(hom-bundl1-bis)
\Phi :  G^n_\eps(\vect) \times \cc^{Nr} \to E
\end{equation}
with the following properties:

\smallskip\sli $\Phi_p: \{p\}\times \cc^{Nr} \to E_p$ is surjective for 
every point $p\in G^n_\eps(\vect)\setminus D$, where $D$ is 
an 

\quad analytic set in $G^n_\eps(\vect)$ of codimension $\ge 2$;

\smallskip\slii for any holomorphic section $\bfs$ of $G^n_\eps(\vect)
\times \cc^{Nr} $ its image $\Phi (\bfs)$ belongs to $L^2(G^n_\eps(\vect),
E)$. 
\end{lem}
\proof Here $\Phi_p$ stands for the restriction of $\Phi$ to the 
fibre $\cc^{Nr}_p\deff\{p\}\times \cc^{Nr}$ over $p$ of the trivial
bundle $G^n_\eps(\vect) \times \cc^{Nr}$  and maps $\cc^{Nr}_p$ to $E_p$. 
Notice that we can find a holomorphic (in fact affine) function $f$ vanishing 
at $p$ and not vanishing on $Y\setminus G^n_\eps(\vect)$. Namely, in 
coordinates $z_1,..,z_n$ used in the definition of $G^n_\eps$, for $p=(z_1^0,...,z_{n-1}^0,0)\in Y\cap G^n_\eps(\vect)$ one can set $f(z) = z_1-z_1^0$. For $p\not\in Y$ one can set $f(z) = z_n-z_n^0$ since $z_n^0\not=0$ 
for such $p$. Using 
Theorem \ref{L2-section} we can find $L^2$ holomorphic sections $\bfs_1, 
\dots, \bfs_r \in  L^2(G^n_\eps(\vect) , E) \cong L^2\left(G^n_\eps (\vect),
E\otimes \Lambda^{n,0}(\cc^n)\right)$, which generate the fibre of 
$E$ at $p$ for a fixed in advance shrunken $G^n_\eps(\vect)$.  Here
$L^2$ is understood with respect to the usual K\"ahler form on $\cc^n$
(its restriction to $G^n_\eps (\vect)$ is not complete!), and the
given metric with Nakano semi-positive curvature on $E$. Consider the 
following morphism of bundles 
\[
\begin{array}{ll}
\Phi_1 : &\calo(G^n_\eps(\vect))^r \,\, \longrightarrow \,\, \calo(G^n_\eps(\vect), E) 
\\[5pt]
&(f_1, \dots, f_r) \longmapsto \sum_{i=1}^r f_i \bfs_i.
\end{array}
\]
If sections $\bfs_1,...,\bfs_r$ generate every fibre $E_z$ of $E$ for
$z\in G^n_\eps(\vect)$ 
then we can take $\Phi = \Phi_1$ as our map. Both items (\sli and (\slii will be satisfied due to the choice of $\bfs_1,...,\bfs_r$ and the fact that they are 
in $L^2(G^n_\eps(\vect) ,E)$.  Otherwise, consider the following set
\[
\widetilde{C} = \left\{ z \in G^n_\eps(\vect) \mid \bfs_1(z), \dots, \bfs_r(z) ~  
\text{do not generate} ~ E_z\right\}.
\]

This is an analytic set in $G^n_\eps(\vect)$. Let $C$ be the union of all components 
of $\widetilde{C}$ of pure 
codimension $1$. By Remmert-Thullen-Stein theorem, see Theorem 2.6 in 
\cite{Si2}, the closure $\bar C$ of $C$ is a hypersurface in 
$\widehat{G}^n_\eps$. Therefore, shrinking $G^n_\eps(\vect)$, if 
necessary, we can assume that 
$\bar C$ has a finite number of irreducible components, say $C_1, \dots, 
C_k$. For each $1 \leq j \leq k$ fix a point $p_j \in C_j\cap 
G^n_\eps (\vect)$. Using the same arguments as above we can find holomorphic sections $\bfs_1^{(j)}, \dots, \bfs_r^{(j)} \in L^2 (G^n_\eps(\vect), E)$ that generate the fibre $E_{p_j}$. 
Consider the mappings
\[
\begin{array}{ll}
\Phi_{j+1} : &\calo(G^n_\eps(\vect))^r \longrightarrow \calo(G^n_\eps(\vect), E) 
\\[5pt]
&(f_1, \dots, f_r) \longmapsto \sum_{i=1}^r f_i \bfs_i^{(j)}
\end{array}
\]
and set 
\[
\Phi = \Phi_1 \oplus \dots \oplus \Phi_{k+1}.
\]
$\Phi$ can be non-surjective at most over a codimension $\ge 2$ set of 
points $D$ in $G^n_\eps(\vect)$. And notice that all sections 
$\Phi (\bfs)\in \calo (G^n_\eps(\vect), E)$  for $\bfs\in \calo (G^n_\eps(\vect), \calo^r)$ are $L^2$-integrable in a shrunken 
$G^n_\eps(\vect)$ as required. Here we use also the fact that holomorphic 
functions extend from $G^n_\eps (\vect)$ to $\widehat{G}^n_\eps(\vect) \deff 
\Delta^n \cap \widehat{P}^n_\eps(\vect)$.

\smallskip\qed

\smallskip\noindent{\slsf Step 2. Improving the surjectivity of $\Phi$. }
We assume that $G^n_\eps(\vect)$ and therefore $\widehat{G}^n_\eps(\vect)$ are 
appropriately shrunken in such a way that the conclusion of Lemma 
\ref{surj-map1-bis} holds. Consider the holomorphic mapping 
\begin{equation}
\eqqno(ker-map-bis)
\begin{array}{ll}
K : &G^n_\eps(\vect) \setminus D\longrightarrow Gr((N-1)r, Nr) \\[4.5pt]
&z \longmapsto \ker ~\Phi_z \subset \cc^{Nr}.
\end{array}
\end{equation}
Since the Grassmanian $Gr((N-1)r, Nr) $ is a projective 
variety holomorphic  mappings with values in $Gr((N-1)r,Nr)$ extend
meromorphically across analytic sets of codimension $\ge 2$. Therefore
$K$ extends meromorphically to the whole of $G^n_\eps(\vect) $. After that we
extend $K$ to $\widehat{G}^n_\eps(\vect) $ by the Thullen extension theorem 
for meromorphic functions. Denote this extension still as $K$ and by 
$A$ denote the set of its indeterminacy points. This $A$ is an analytic 
set in $\widehat{G}^n_\eps(\vect) $ of codimension $\ge 2$. 

\begin{lem}
\label{surj-map2-bis}
Mapping $\Phi_p: \cc^{Nr}_p \to E_p$ is surjective for every 
point $p\in G^n_\eps(\vect) \setminus A$.
\end{lem}
\proof  Still denoting by $D$ the set of points $p\in G^n_\eps(\vect)$ 
such that $\Phi_p$ is not surjective let us prove that $D\subset A\cap G^n_\eps(\vect)$. If not take a point $p\in D\setminus A$ and find a
polydisk neighbourhood $W\ni p$ such that  $K$ is holomorphic in $W$ and,
moreover, such that the mapping $\Phi$ is surjective over $W\setminus 
D$. I.e., $\Phi_z:\cc^{Nr}_z\to E_z$ has rank $r$ for all $z\in W\setminus D$. 
Let $e_1,...,e_{Nr}$ be the standard basis of $\cc^{Nr}$, defining
naturally the frame of $\calo (G^n_\eps(\vect))^{Nr}$. Denote this frame 
as $\bfe_1,..., \bfe_{Nr}$. Fix some frame $\bftheta_1,...,\bftheta_r$ of $E|_W$. 
In these frames mapping $\Phi_z$ writes naturally as a matrix valued function 
\[
\Phi (z)\left(\{f_i\}_{i=1}^{Nr}\right) = 
\]

\begin{equation}
\eqqno(matrix-phi-bis)
= \left(
\begin{matrix}
\sigma_{11} &\sigma_{12} & ...& \sigma_{1Nr}\cr
\sigma_{21} &\sigma_{22} & ...& \sigma_{2Nr} \cr 
. & . & ... & .\cr
\sigma_{r1} &\sigma_{r2} & ...& \sigma_{rNr}
\end{matrix}
\right)
\left(
\begin{matrix}
f_1 \cr
f_2 \cr
..\cr 
f_{Nr}
\end{matrix}
\right)
=
\left(
\begin{matrix}
\sigma_{11}f_1 + \sigma_{12}f_2 + ... + \sigma_{1Nr}f_{Nr}\cr
\sigma_{21}f_1 + \sigma_{22}f_2 + ... + \sigma_{2Nr}f_{Nr} \cr
................................\cr
\sigma_{r1}f_1 + \sigma_{r2}f_2+  ... + \sigma_{rNr}f_{Nr}
\end{matrix}
\right),
\end{equation}

\noindent where $\Phi_z(\bfe_k) = \sigma_{1k}\bftheta_1 + \sigma_{2k}
\bftheta_2 + ...+ \sigma_{rk}\bftheta_r$. I.e., the columns of this
matrix are formed from the coefficients of sections $\bfs_i$, $i=1,...,Nr$, 
in the frame $\bftheta_1,...,\bftheta_r$, this is because $\Phi(\bfe_k) = 
\bfs_k$. Here we denoted the matrix of $\Phi_z$ as $\Phi (z)$. Its rows 
are vectors in $\cc_z^{Nr}$ generating the annihilator
$\ker\Phi_z^{(0)}$ of $\ker \Phi_z$ in $\cc_z^{Nr}$. Mapping 
\[
K^{(0)}:z\longrightarrow \ker \Phi_z^{(0)} \in Gr (r,Nr)
\]
is well defined for all $z\in W\setminus D$, where $K(z)$ is well defined.  
Moreover they both are holomorphic on the same open subset 
of $W$ where they both are well defined. Let us prove the meromorphic 
extendability of $K^{(0)}$ to the whole of $W$.

\smallskip Denote by $C_r^{Nr}$ the binomial coefficient and consider the 
Pl\"ucker embedding, see \cite{GH}: 
\begin{equation}
\eqqno(plucker-bis)
P :  Gr (r,Nr) \longrightarrow \pp^{C_r^{Nr}-1}.
\end{equation}
Composition $P\circ K^{(0)}$ is a holomorphic map of $W\setminus D$
to $\pp^{C_r^{Nr}-1}$. In local frames as above our composition writes as 
\begin{equation}
\eqqno(plucker-1-bis)
P\circ K^{(0)}: z \longrightarrow [\Delta_{1...,r}:...:
\Delta_{(N-1)r+1...Nr}],
\end{equation}
where $\Delta_{i_1...i_r}$ runs over all minors of the matrix $\Phi_z$ of 
size $r$. Another description of this mapping for $z\in W\setminus D$ is 
\begin{equation}
\eqqno(plucker-2-bis)
P\circ K^{(0)}: z \longmapsto v_1(z)\wedge ...\wedge v_r(z),
\end{equation}
 where $v_1,...,v_r$ span 
$\ker\Phi (z)^{(0)}$. For $z\in W\setminus D$ as these
$v_1,...,v_r$ one can take the rows of the matrix $\Phi (z)$, and this is 
the same as \eqqref(plucker-1-bis).  For a given point $z\in W$ such that not 
all minors vanish at $z$ the composition $P\circ K^{(0)}$is holomorphic 
in a neighbourhood of such $z$,  which gives the holomorphicity
of $K^{(0)}$ in this neighbourhood. 

\smallskip If $p\in D$, \ie $\Phi_p$ is not surjective, our minors vanish
at $p$. Since $\Phi_z$ is surjective for $z$ outside of codimension $\ge 2$ 
they cannot vanish along a germ of a hypersurface through $p$. 
Denote by $PL(r, Nr)$ the image of $Gr(r,Nr)$ under the Pl\"ucker 
embedding \eqqref(plucker-bis). This is a smooth projective subvariety of 
$\pp^{C_r^{Nr}-1}$. The fact that $p\in D\setminus A$ means that 
$P\circ K^{(0)}$ is holomorphic near $p$ as a mapping to the projective space,
but the rows of the matrix $\Phi (p)$ are linearly dependant at $p$.
I.e., all minors in representation \eqqref(plucker-1-bis) vanish at $p$.
But a meromorphic mapping $f:\Delta^n\to \pp^M$, $M = C_r^{Nr}$ 
in our case, given in the form 
\[
f(z) = [\phi_0(z):...:\phi_M(z)]
\]
($\phi_j$ are minors in our case) with $\phi_j$ vanishing at zero but 
having no common non-invertible 
factor as germs in $_n\calo_0$, cannot be holomorphic near the origin,
see the second item in the Remark \ref{zero-ind}.
Applying this to $f=P\circ K^{(0)}$ we get a contradiction.


\smallskip\qed

\smallskip 
Let us further improve the surjectivity of $\Phi$.
\begin{lem}
\label{surj-map3-bis}
Morphism of bundles $\Phi : G^n_\eps(\vect) \times \cc^{Nr} \to E$ can be 
modified in such a way that, after shrinking, $\Phi_p : \cc^{Nr}_p \to E_p$
will be surjective for every point $p\in G^n_\eps(\vect) $.
\end{lem}
\proof Being an analytic set in $\widehat{G}^n_\eps(\vect)$ our $A$ has, after
a shinking, only finitely many irreducible components. For every
such component which intersects $G^n_\eps(\vect)$, take a point $p$ on 
this intersection and repeat the construction from Lemma \ref{surj-map1-bis} 
for this point to make $\Phi_p$ surjective at $p$. 

\smallskip\qed

\smallskip\noindent{\slsf Step 3. Extension of $E$ to $\widehat{G}
^n_\eps(\vect)$}. By Lemma \ref{surj-map3-bis} morphism of bundles 
$\Phi : G^n_\eps(\vect) \times \cc^{Nr}\to E|_{G^n_\eps(\vect)}$ 
is surjective and therefore induces an isomorphism of bundles  
\begin{equation}
\eqqno(iso-bnd)
\widetilde{\Phi } : G^n_\eps(\vect) \times \cc^{Nr}/\ker \Phi 
\to E|_{G^n_\eps(\vect)}.
\end{equation}

Denote the sheaf of sections of $E|_{G^n_\eps(\vect)} $ by $\cale$.


\begin{lem}
\label{surj-lem4-bis}
There exists a torsion free coherent analytic sheaf $\wt\cale$ on 
$\widehat{G}^n_\eps(\vect)$ such that $\cale$ is canonically isomorphic 
to the restriction  of $\wt\cale$ to $G^n_\eps(\vect)$.
\end{lem}
\proof Let $\pi:\widehat{V}\to \widehat{G}^n_\eps(\vect)$ be a proper 
modification such that $\widehat{K} \deff K\circ \pi : \wt V \to 
Gr ((N-1)r,Nr)$ is holomorphic. Here $K$, as before, stands for the 
meromorphic extension to $\widehat{G}^n_\eps(\vect)$ of the "kernel
map" $K$ defined in \eqqref(ker-map-bis).  Denote by $\widehat{\cale}$
the sheaf of holomorphic sections of the bundle $\hat V\times 
\cc^{Nr}$ and by $\hat\calk$ the sheaf of section of its subbundle 
$\im \hat K$. By Grauert's direct image theorem the direct images 
$\pi_*\hat\cale$ and $\pi_*\hat\calk$ are coherent and therefore 
coherent is their quotient $\wt\cale \deff\pi_*\hat\cale /\pi_*
\hat\calk$. This is our extension. Since the modification $\pi$ can
be chosen such that it blows-up only the indeterminacy points
of $\widehat{K}$, and these points are contained in $Y\setminus 
G^n_\eps(\vect)$, we see that $\wt\cale |_{G^n_\eps(\vect)}$
and $\cale $ are canonically isomorphic. We had denote 
in \eqqref(iso-bnd) by $\widetilde{\Phi} : \widetilde{\cale}|_{G^n_\eps(\vect)}
\to \cale$ this isomorphism. Indeed, notice that $\widetilde{\cale}|_{G^n_\eps(\vect)}$ is nothing but $ G^n_\eps(\vect) \times \cc^{Nr}/\ker \Phi $.

\smallskip As the result we got a coherent torsion free analytic sheaf 
$\widetilde{\cale}$ on $\widehat{G}^n_\eps(\vect)$ such that 
$\widetilde{\cale}|_{G^n_\eps(\vect)}$ is locally free, \ie 
is a bundle, and $\widetilde{\cale}|_{G^n_\eps(\vect)}$
is isomorphic to the sheaf $\cale$ of section of our bundle $E$. 
$\widetilde{\cale}$ is torsion free since the set of points where
a sheaf has torsion is a hypersurface and therefore should intersect 
$G^n_\eps (\vect)$.

\smallskip\qed


\smallskip\noindent{\slsf Step 4. Integrability condition.} 
Let us see that holomorphic sections of $\widetilde{\cale}
|_{G^n_\eps(\vect)}$ are mapped by the constructed isomorphism
$\widetilde{\Phi}$ to holomorphic $L^2_\loc$-sections  of $E$ over 
$G^n_\eps(\vect)$. Here by $L^2_{loc}$-section of $E$ we mean a holomorphic section
$\bfs$ of $E$ over $G^n_\eps(\vect)$ such that for any  compact 
$K\comp \widehat{G}^n_\eps(\vect) $ one has that $\bfs|_{K\cap G^n_\eps(\vect) }
\in L^2(K\cap G^n_\eps(\vect),E)$. 

\smallskip
First notice that if $\bfe\in \calo (G^n_\eps(\vect) )^{Nr}$ then $\bfs\deff
\Phi (\bfe) \in L^2_\loc(G^n_\eps(\vect) , E)$, since it is by construction 
a linear combination of $L^2_\loc$-sections with coefficients holomorphically
extending to $\widehat{G}^n_\eps (\vect)$. Now let $\wt\bfe$ be a holomorphic 
section of $\wt\cale|_{G^n_\eps (\vect)}$, $\wt\bfe$ extends to a holomorphic 
section, denote it still by $\wt\bfe$, of $\wt\cale$ over $\widehat{G}
^n_\eps(\vect)$. This is because $\widetilde{\cale}$ is torsion free and 
therefore is normal.

\smallskip  Let $\{U_\alpha\}$ 
be a Stein covering of $\widehat{G}^n_\eps (\vect)$ and let $\hat e_\alpha$
be sections of $\pi_*\hat\cale$ over $U_\alpha$ defining $\wt \bfe$,
\ie $\hat e_\alpha - \hat e_\beta \in H^0(U_\alpha\cap U_\beta ,
\pi_*\hat\calk)$. This gives us the cocycle $\{\hat e_\alpha - \hat e_\beta\}
\in H^1(\widehat{G}^n_\eps (\vect), \pi_*\calk)$. The latter cohomology
group vanishes and we get $\{e_\alpha\}\in H^0(\widehat{G}^n_\eps (\vect), \pi_*\hat\calk)$ such that $\delta \{e_\alpha\} = \{\hat e_\alpha - \hat e_\beta\}$. As the result we obtain a global section $\bfe$  of 
$\pi_*\widehat{\cale}$ representing $\wt \bfe$. Since 
$\wt\Phi(\wt\bfe|_{G^n_\eps(\vect)})=\Phi(\bfe|_{G^n_\eps(\vect)})$ 
we see that $\wt\Phi(\wt\bfe|_{G^n_\eps(\vect)})$ is locally integrable
as stated.




\smallskip Finally, replace $\widetilde{\cale}$ by its second dual 
$\widetilde{\cale}^{**}$ and get the reflexive extension of $\cale$.

\begin{rema} \rm
\label{dim-2-bis}
Notice that considering $G^2_\eps$ on the place of $G^n_\eps(\vect)$
in the previous considerations we obtain the  proof (the local version 
of) our theorem in dimension two. Since our
extended sheaf is reflexive in dimension two this is a sheaf of 
holomorphic sections of a bundle. Therefore
in dimension two $E$ extends to $\Delta^2$ as a bundle.
\end{rema}

\medskip\noindent{\slsf Step 5. Proof of the local case of
Theorem \ref{thul-posit} in dimension $\ge 3$.} We start with the following
lemma.

\begin{lem}
\label{ext-refl-morph-bis}
Let $X$ be a complex manifold, $G$ be a Thullen domain of $X$ and $\calf, \calg$ be two reflexive coherent analytic sheaves on X.  Then
\begin{equation}
    \Gamma(G, Hom_{\calo}(\calf, \calg)) \simeq \Gamma(X, Hom_{\calo}(\calf, \calg)).
\end{equation}
\end{lem}

\proof Let $s \in \Gamma(G, Hom_{\calo}(\calf, \calg))$. Since $\calf$ and 
$\calg$ are reflexive, there exists an analytic subset $A$ of $X$  
of codimension $\ge 3$ such that 
$\calf$ and $\calg$ are locally free over $X \smallsetminus A$. Then there 
exists a unique holomorphic section $\ti s \in \Gamma(X \smallsetminus A, 
Hom_{\calo}(\calf, \calg))$ such that $\ti s|_{G\smallsetminus A} = 
s|_{G\smallsetminus A}$. Indeed, holomorphic sections of a holomorphic 
vector bundle admit the Thullen-type extension (where the bundles are defined). 
Since $\calf$ and $\calg$ are reflexive, so $Hom_{\calo}(\calf, \calg)$ is reflexive. In particular, $Hom_{\calo}(\calf, \calg)$ is normal, i.e. for every open subset $U \subset X$ and every analytic subset $Z$ such that $codim_XZ \geq 2$, 
\begin{equation}
    \Gamma(X \smallsetminus Z, Hom_{\calo}(\calf, \calg)) \simeq \Gamma(X, Hom_{\calo}(\calf, \calg)),
\end{equation}
(see \cite{Kob} for details). Hence, there exists a unique section $\sigma \in \Gamma(X, Hom_{\calo}(\calf, \calg))$ such that 
$\sigma|_{X\setminus A} = \ti s$ and so $\sigma|_{G\smallsetminus A}= s|_{G\smallsetminus A}$. Using the normality of $Hom_{\calo}(\calf, \calg)$ again, we obtain that $\sigma|_{G}= s$.

\smallskip\qed

\begin{prop}
\label{local-ext-bis}
The sheaf of holomorphic sections $\cale$ of a Nakano semi-positive bundle 
$E$ over $G^n_{\eps}$ admits a reflexive coherent analytic extension 
$\wt\cale $ to $\Delta^n$. Moreover, the sheaf isomorphism $\wt\Phi:
\wt\cale|_{G^n_\eps} \to \cale$ maps holomorphic sections of $\wt\cale$ 
over $G^n_\eps$ to holomorphic section of $\cale$ which are $L^2$-bounded
over $K\cap G^n_\eps$ for every compact $K\comp\Delta^n$.
\end{prop}

\proof For every unit vector $\vect\in \cc^{n-1}$ let $\widetilde{\cale} 
(\vect)$ be the reflexive coherent analytic sheaf extending 
$\cale|_{G^n_\eps (\vect)}$ to $\widehat{G}^n_\eps (\vect)$. Let us see 
that for two unit vectors $\vect_1$, $\vect_2$ as above the identity 
morphism of $\cale|_{G^n_\eps (\vect_1)\cap G^n_\eps (\vect_2)}$
onto itself extends to the sheaf isomorphism between restrictions 
$\widetilde{\cale} (\vect_1)|_{\hat G^n_\eps (\vect_1)\cap \hat G^n_\eps 
(\vect_2)}$ and $\widetilde{\cale} (\vect_2)|_{\hat G^n_\eps (\vect_1)\cap 
\hat G^n_\eps (\vect_2)}$. Indeed, when constructing the extension 
$\widetilde{\cale} (\vect_i)$ of $\cale |_{G^n_{\eps} (\vect_i)}$ over 
$\widehat{G}^n_\eps (\vect_i)$ we constructed also an isomorphisms 
$\widetilde{\Phi}(\vect_i) : \wt\cale |_{G^n_{\eps} (\vect_i)} \to \cale |_{G^n_{\eps} (\vect_i)}$ for $i = 1, 2$. Let the isomorphism 
$\psi_{\vect_1,\vect_2} $  between $\wt\cale (\vect_1)|_{G^n_\eps (\vect_1)
\cap G^n_\eps (\vect_2)}$ and $\wt\cale (\vect_2)|_{G^n_\eps (\vect_1)\cap 
G^n_\eps (\vect_2)}$  be defined by the following diagram 


\begin{equation}
\begin{CD}
\eqqno(globa-on-bis)
\cale|_{G^n_\eps (\vect_1)\cap G^n_\eps (\vect_2)} @= \cale|_{G^n_\eps 
(\vect_1)\cap G^n_\eps (\vect_2)}\\
@A\widetilde{\Phi}(v_1)A\simeq A @A\simeq A\widetilde{\Phi}(v_2)A\\
\widetilde{\cale} (\vect_1)|_{G^n_\eps (\vect_1)\cap G^n_\eps (\vect_2)} 
@>\psi_{\vect_1,\vect_2}>> \widetilde{\cale} (\vect_2)|_{G^n_\eps (\vect_1)\cap G^n_\eps (\vect_2)}
\end{CD}
\end{equation}

\smallskip Since $\hat G^n_\eps (\vect_1)\cap \hat G^n_\eps (\vect_2) $
is the envelope of holomorphy of $G^n_\eps (\vect_1)\cap G^n_\eps (\vect_2) $
Lemma \ref{ext-refl-morph-bis} ensures that the isomorphism $\psi_{\vect_1,
\vect_2} $ extends over 
$\hat G^n_\eps (\vect_1)\cap \hat G^n_\eps (\vect_2) $. And therefore 
$\wt\cale$ glues to a sheaf over $\hat G^n_\eps (\vect_1)\cup 
\hat G^n_\eps (\vect_2) $ and $\wt\Phi (\vect_1), \wt\Phi (\vect_2)$ 
to a isomorphism $\wt\Phi (\vect_1,\vect_2): \wt\cale|_{G^n_\eps (\vect_1)
\cup G^n_\eps (\vect_2)  } \to \cale|_{G^n_\eps (\vect_1)
\cup G^n_\eps (\vect_2)  }$. By construction, these isomorphisms verify 
the cocycle conditions 
\begin{equation}
\psi_{\vect_1, \vect_1} = \id, \quad \psi_{\vect_2, \vect_3} \circ \psi_{\vect_1, \vect_2} = \psi_{\vect_1, \vect_3}.
\end{equation}
Hence, up to isomorphisms, there exists a unique analytic
sheaf $\wt \cale$ over $\Delta^n$ such that $\wt \cale |_{\hat G^n_{\eps}(\vect)} \simeq \wt\cale(\vect)$ for every unit vector $\vect$ of $\cc^{n-1}$. Moreover, 
$\wt\cale$ is coherent and reflexive.

\smallskip Now fix a holomorphic section $\tilde{e}$ of $\tilde\cale$ on $G^n_\eps$ and a compact $K$ of $\Delta^n$. We want to prove that $\tilde\Phi(\tilde e)$ is $L^2$-bounded over $K \cap G^n_\eps$. Since $K$ is compact, we can find $x_1, \dots, x_m \in K$ and $r_1, \dots, r_m > 0$, where $m \in \nn$, such that for each $j \in \{1,\dots, m\}$ there exists a unit vector $\vect_j \in \cc^{n-1}$ such that the closed euclidean ball $\bar B(x_j, r_j)$ is contained in $\hat G^n_{\eps}(\vect_j)$ and verifying $K \subset \bigcup_{j=1}^m \bar B(x_j, r_j)$. Hence, we have
\begin{equation*}
    \| \tilde\Phi(\tilde e)\|_{L^2(K\cap G^n_\eps)} \leq \sum_{j=1}^m \| \tilde\Phi(\tilde e)\|_{L^2(\bar B(x_j, r_j)\cap G^n_\eps(\vect_j))} < +\infty,
\end{equation*}
the last inequality is obtained using the integrability condition obtained in 
{\slsf Step 4}. Hence, $\tilde\Phi(\tilde e)$ is $L^2$-bounded over $K \cap G^n_\eps$.

\smallskip\qed

\newprg[EXT.global-bis]{Proof of the Theorem \ref{thul-posit}, global case} Now we prove the Theorem \ref{thul-posit}. Let $X$ be a complex manifold of dimension $n \geq 2$, $Y \subset X$ be a submanifold of codimension $1$, $G \subset X$ an open set  which contains $X\setminus Y$ and intersects every branch of $Y$.
Let $(E, h)$ be a holomorphic hermitian bundle of rank $r$ over $G$ with
Nakano semi-positive curvature. 

\smallskip First, notice that we can assume that $Y$ is irreducible. Indeed, 
since $G$ intersects each branch of $Y$, we can extend $E$ across each branch
and then "glue" the obtained extensions into a reflexive coherent analytic 
sheaf over $X$.

\smallskip Fix $x\in Y\cap G$. For every point $y$ in $Y\smallsetminus G$,
there exists a regular real analytic path $\gamma : [0, 1] \to Y$ such that
$\gamma(0) = x$ and $\gamma(1) = y$. Hence, there exists an embedded analytic disc $A$ in $Y$ which contains $x$ and $y$. By the Royden's lemma, we can embed a $n$-dimensional polydisc $U_y$ centred at $x$ in such a way that $U_y \cap Y \simeq \Delta^{n-1} \times \{0\}$ (we use Royden's lemma two times, the first time we embed a $(n-1)$-polydisc neighbourhood $P$ of $A$ in $Y$, and then we use it a second time to find a $n$-dimensional polydisc neighbourhood of $P$ in $X$).
We denote $U_y \cap G$ by $G_y$.

\begin{figure}[h]
\centering
\scalebox{.7}{\includegraphics[width=5.5in]{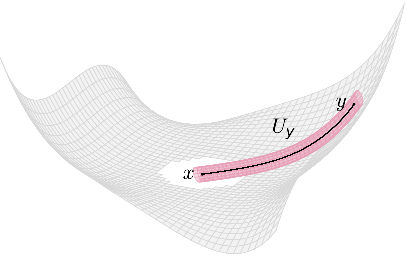}}
\caption{}
\label{path}
\end{figure}




\smallskip By Lemma \ref{local-ext-bis}, $E|_{G_y}$ extends to $U_y$, 
denote by $\wt\cale (y)$ this extension. For every two points $y_1, y_2$ of $Y$, denote by $\psi_{y_1, y_2}$ the sheaf isomorphism between $\wt\cale (y_1)|_{G_{y_1}\cap G_{y_2}}$ and $\wt\cale (y_2)|_{G_{y_1}\cap G_{y_2}}$. By 
construction these isomorphisms verify the cocycle conditions 
\begin{equation}
    \psi_{y_1, y_1} = \id, \quad \psi_{y_2, y_3} \circ \psi_{y_1, y_2} = \psi_{y_1, y_3}.
\end{equation}
Now, Lemma \ref{ext-refl-morph-bis} ensures that these morphisms uniquely extend over $U_{y_1}\cap U_{y_2}$, and that these extensions still satisfy the cocycle conditions (in particular, the extensions are isomorphisms). Hence, up to isomorphisms, there exists a unique analytic sheaf $\wt \cale$ over $X$ such 
that $\wt \cale |_{U_y} \simeq \wt\cale (y)$ for every point $y$ of $Y$. 
Moreover, $\wt\cale$ is coherent and reflexive.

\smallskip Furthermore, using the same arguments as in the local case, we can prove that for every holomorphic section $\ti s$ of $\wt \cale |_G$  its $\wt\Phi$-image  
is $L^2$-bounded on compacts in $X$, \ie $\wt \Phi (\ti s|_{K\cap G})
\in L^2(K\cap G,E)$ for every compact $K\comp X$.

\smallskip\qed

\ifx\undefined\bysame
\newcommand{\bysame}{\leavevmode\hbox to3em{\hrulefill}\,}
\fi

\def\entry#1#2#3#4\par{\bibitem[#1]{#1}
{\textsc{#2 }}{\sl{#3} }#4\par

\end{document}